\documentclass[12pt]{article}
\usepackage[english]{babel}
\usepackage{amsmath,amsfonts}
\usepackage{amsthm}
\usepackage{amssymb}
\usepackage{amscd}
\usepackage{fancyhdr}
\usepackage{graphicx}
\usepackage{setspace}
\usepackage{verbatim}
\usepackage{tikz-cd}
\usepackage{pdfsync}
\usepackage{stmaryrd}
\usepackage{calrsfs}
\usepackage{mathrsfs}
\usepackage{cite}
\usepackage[a4paper,left=1cm,right=1cm,top=3cm,bottom=1.5cm]{geometry}
\usetikzlibrary{positioning, arrows}
\title{}
\author{}

\theoremstyle{definition}
\newtheorem{Theo}{Theorem}[section]
\newtheorem{Rmk}[Theo]{Remark}
\newtheorem{Ex}[Theo]{Example}

\newtheorem{Def}[Theo]{Definition}
\newtheorem{Lemma}[Theo]{Lemma}
\newtheorem{Cor}[Theo]{Corollary}

\newtheorem{Prop}[Theo]{Proposition}

\theoremstyle{plain}

\DeclareMathOperator{\id}{id}

\newcommand{\overbar}[1]{\mkern 1.5mu\overline{\mkern-1.5mu#1\mkern-1.5mu}\mkern 1.5mu}

\renewcommand{\bar}{\overbar}
\renewcommand{\tilde}{\widetilde}

\begin{document}
	\author{Shanxiao Huang}
	\date{}
    \title{Coordinate Transformation in Faltings' Extension\footnote{This is author's master thesis in 2018.}}
    \maketitle
	\tableofcontents
	\vspace {25em}
	\pagebreak
	{ \ \ }\\
	\textbf{Abstract.}Analogue to Fontaine's computation for $\Omega_{\bar{\mathbb{Z}}_p/\mathbb{Z}_p}$, we compute the structure of $\Omega_{\mathcal{O}_{\bar{K}_0}/\mathcal{O}_{K_0}}$ (here $K_0$ is the completion of $\mathbb{Q}_p(T)$ at place $p$) and prove that $p^{1-1/p^n}\mathrm{d}p^{1/p^n}$, $T^{1-1/p^n}\mathrm{d}T^{1/p^n}$ and $S^{1-1/p^n}\mathrm{d}S^{1/p^n}$ are linearly dependent (Here $S := 1-T$). The main aim of this article is to find the linear equations for these three differential forms. Then we define a map which is called "differential version" of Fontaine's map to express the equations in a computable way. Finally, we prove that the coefficients in the equation can be expressed in some polynomial forms and compute some examples.\\
	
	
	\setcounter{section}{-1}
	
	\section {Introduction}
	
	In \cite{Faltings}, Faltings computed the structure of $\Omega_{\bar{\mathbb{Z}}_p/\mathbb{Z}_p}$, namely
	$$\Omega_{\bar{\mathbb{Z}}_p/\mathbb{Z}_p}\cong (\mathbb{Q}_p/\mathbb{Z}_p)\otimes\bar{\mathbb{Z}}_p.$$
	It is interesting to consider a nontrivial generalization by introducing an additional variable. More precisely, let $K_0$ be the $p$-adic completion of $\mathbb{Q}_p(T)$ and $\bar{K}_0$ the algebraic closure of $K_0$. Our first goal is to compute $\Omega{\mathcal{O}_{\bar{K}_0}/\mathcal{O}_{K_0}}$. Conceptually, the pair $(\bar{\mathbb{Z}}_p, \mathbb{Z}_p)$ corresponds to a geometric object of dimension one. Adding one free variable transforms the pair $(\mathcal{O}_{\bar{K}_0}, \mathcal{O}_{K_0})$ into a dimension two setting. The following result is obtained:
	
	\begin{Prop} [Cor. \ref{Cor.KahlerDiff}]
		The map
		\begin{align*}
			\lambda:& (\mathbb{Q}_p/\mathbb{Z}_p)^{\oplus 2}\otimes \mathcal{O}_{\bar{K}_0}\rightarrow \Omega_{\mathcal{O}_{\bar{K}_0}/\mathcal{O}_{K_0}}\\
			&\frac{1}{p^m}\otimes \alpha +\frac{1}{p^n}\otimes \beta \mapsto \alpha p^{1-1/p^m}\mathrm{d}p^{1/p^m}+\beta T^{1-1/p^n}\mathrm{d}T^{1/p^n} 
		\end{align*}
		is an isomorphism.
	\end{Prop}

    One can observe that $(p^{1-1/p^n}\mathrm{d}p^{1/p^n},T^{1-1/p^n}\mathrm{d}T^{1/p^n})$ forms a basis of $\Omega_{\mathcal{O}_{\bar{K}_0}/\mathcal{O}_{K_0}}[p^n]$ (the $p^n$-torsion part of $\Omega_{\mathcal{O}_{\bar{K}_0}/\mathcal{O}_{K_0}}$) as free a $\mathcal{O}_{\bar{K}_0}/(p^n)$-module. Note that the automorphism 
    $$K_0\rightarrow K_0, T \mapsto 1-T$$
    can be extend to an automorphism of $\bar{K}_0$, which implies $1-T$ can replace $T$ in the proposition above. Thus the pair $(p^{1-1/p^n}\mathrm{d}p^{1/p^n},(1-T)^{1-1/p^n}\mathrm{d}(1-T)^{1/p^n})$ also forms a basis of $\Omega_{\mathcal{O}_{\bar{K}_0}/\mathcal{O}_{K_0}}[p^n]$. Consequently, the tuple $(p^{1-1/p^n}\mathrm{d}p^{1/p^n},T^{1-1/p^n}\mathrm{d}T^{1/p^n},(1-T)^{1-1/p^n}\mathrm{d}(1-T)^{1/p^n})$ are linearly dependent. Our next goal is to determine their linear equation for each $n$.\\
    
    Before stating the result, we can rephrase our problem more elegantly. Denote
    $$\mathbb{T}:= \varprojlim_{x\mapsto px} \Omega_{\mathcal{O}_{\bar{K}_0}/\mathcal{O}_{K_0}}[p^n],$$
    and denote
    $$\mathrm{d}_{\theta}[X^{\flat}] := (X^{1-1/p^n} \mathrm{d} X^{1/p^n} )_{n\in \mathbb{N}} \in \mathbb{T}$$
	for $X\in \mathcal{O}_{\bar{K}_0}$ (note that $pX^{1-1/p^n} \mathrm{d} X^{1/p^n} = X^{1-1/p^{n-1}} \mathrm{d} X^{1/p^{n-1}}$ and see section \ref{Sect.Fantain'sMap} for more details about the notations). Then one has
	\begin{align*}
		\mathcal{O}_{K}\oplus \mathcal{O}_{K} & \xrightarrow{\cong} \mathbb{T}\\
		(\alpha,\beta) & \mapsto \alpha \mathrm{d}_{\theta}[p^{\flat}] + \beta\mathrm{d}_{\theta}[T^{\flat}] \text{ (or } \alpha \mathrm{d}_{\theta}[p^{\flat}] + \beta\mathrm{d}_{\theta}[(1-T)^{\flat}]),
	\end{align*}
	here $K$ is the completion of $\bar{K}_0$.\\
	
	Now the tuple $(\mathrm{d}_{\theta}[p^{\flat}], \mathrm{d}_{\theta}[T^{\flat}], \mathrm{d}_{\theta}[(1-T)^{\flat}])$ is linearly dependent. Our goal is to find $\alpha,\beta,\gamma\in \mathcal{O}_{\bar{K}_0}$ such that the following equation holds 
	$$\alpha \mathrm{d}_{\theta}[p^{\flat}] + \beta \mathrm{d}_{\theta}[T^{\flat}] + \gamma \mathrm{d}_{\theta}[(1-T)^{\flat}]= 0$$

	We express these coefficients using Fontaine's map $\theta: A_{\mathrm{inf}} \rightarrow \mathcal{O}_K$ (the notation $A_{\mathrm{inf}}$ here is a little bit different from the usual one in $p$-adic Hodge theory as we add one free variable, see section \ref{Sect.CoeffFormula} for more details). Roughly speaking, $\theta$ is a homomorphism of $\mathbb{Z}_p$-algebra with principle kernel $([p^{\flat}]-p)$ and $\theta ([T^{\flat}] + [(1-T)^{\flat}] - 1) = 0$. Then there exists an element $A\in A_{\mathrm{inf}}$ such that
	$$ [T^{\flat}] + [(1-T)^{\flat}] - 1 = A ([p^{\flat}]-p),$$
	and the coordinate transformation formula is
	\begin{Prop}[Prop. \ref{Prop.ProofofFormula}]
		Let $\theta, A$ be as above, then one has
		$$ \mathrm{d}_{\theta}[T^{\flat}] + \mathrm{d}_{\theta}[(1-T)^{\flat}]= \theta(A) \mathrm{d}_{\theta}[p^{\flat}]$$
	\end{Prop}
    We develop Fontaione's method and construct a special derivation $$\mathrm{d}_{\theta}: A_{\mathrm{inf}} \rightarrow \Omega_{\mathcal{O}_K/\mathcal{O}_{K_0}} $$ called "differential version" of Fontaine's map, which is the key to prove the proposition above. On the other hand, This proposition solves our original problem. For each $n\in \mathbb{N}$, denote by $a_n$ the image of $\theta(A)$ in $\mathcal{O}_{\bar{K}_0}/(p^n)$, then
    $$ T^{1-1/p^n}\mathrm{d}T^{1/p^n} + (1-T)^{1-1/p^n}\mathrm{d}(1-T)^{1/p^n} = a_np^{1-1/p^n}\mathrm{d}p^{1/p^n}.$$
    
    Moreover, this expression is totally computable and one can find some examples for small $n$ in section \ref{Sect.Examples}. Those concrete computation indicate an surprising phenomenon. That is, one can choose $a_n \in \mathbb{Z}[p^{1/p^n},T^{1/p^n},(1-T)^{1/p^n}]$ even though the ring $\mathbb{Z}[p^{1/p^n},T^{1/p^n},(1-T)^{1/p^n}]$ is much smaller that $\mathcal{O}_{K_0(p^{1/p^n},T^{1/p^n},(1-T)^{p^n})}$. We prove this property by plugging in to the construction of Witt vector and analyzing the calculation process detailed.
    
    \begin{Prop} [Cor. \ref{Cor.Good}]
    	The exists a sequence $(a_1,a_2,\dots)$, where $a_n\in \mathbb{Z}[p^{1/p^{n-1}},T^{1/p^n},S^{1/p^n}]$ such that $$\theta(A) \equiv a_n(\text{mod}\ p^n)$$
    	for each $n\in\mathbb{N}^*$.
    \end{Prop}

    In sections 1-2, we compute structure of $\Omega_{\mathcal{O}_{\bar{K}_0}/\mathcal{O}_{K_0}}$. In sections 3-5, we develop the "Differential version" of Fontaine's Map and use it to prove the coordinate transformation formula. In section 6, we show that the coefficients can be expressed in polynomial forms and provide some examples in section 7. Technical lemmas about cotangent complex are included in the Appendix.\\
    
    The author would like to express his deepest gratitude to his supervisor,  Prof. Peter Scholze, for his unwavering support and guidance for this paper. His patience, expertise, and encouragement have been invaluable in helping the author navigates the challenges of his research and complete this work. Prof. Peter Scholze's insightful feedback and constructive suggestions have greatly contributed to the quality of this paper. The author is profoundly grateful for his mentorship.\\

	\textbf{Notation.}In this article, $p$ is a fixed prime number, and $\mathbb{Q}_p$ denotes the $p$-adic completion of $\mathbb{Q}$, and  its discrete valuation ring is denoted by $\mathbb{Z}_p$. Denote by $\Omega_{B/A}$ the K\"{a}hler differential of a ring map $A\rightarrow B$. For an element $x$ in a field $k$, $x^{1/p^n}$ is inductively 	defined as a certain $p$-root of $x^{1/p^{n-1}}\in\bar{k}$ (i.e. $(x^{1/p^n})^p = x^{1/p^{n-1}}$).  A local field means a field which is complete under some discrete valuation.\\
	
	\section{The Structure of $\Omega_{\bar{\mathbb{Z}}_p/\mathbb{Z}_p}$}
	\Lemma\label{Lem.DirectLimitofKahler}
	Assume $M:\mathcal{I}\rightarrow A\text{-alg}$ is a filtered system of $A$-algebras with a colimit $B$. Then the natural morphism $\phi:\text{colim} (\Omega_{M(i)/A}\otimes_{M(i)} B) \rightarrow \Omega_{B/A}$ is an isomorphism.
	
	\begin{proof}
	Denote $A_i := M(i)$, the commutative diagram:
	\begin{equation*}\begin{tikzcd}
	A \arrow{rr}& & A_i \arrow{rd}{g_i}\arrow{rr}{f_i} & &A_{j}\arrow{ld}{g_{j}}\\
	& & &B
	\end{tikzcd}
	\end{equation*}
	induces the following commutative diagram:
	\begin{equation*}\begin{tikzcd}
	\Omega_{A_i/A}\otimes_{A_i}B \arrow{rd}{g_{i,\ast}}\arrow{rr}{f_{i,\ast}\otimes B} & &\Omega_{A_{j}/A}\otimes_{A_{j}}B\arrow{ld}{g_{j,\ast}}\\
	&\Omega_{B/A}
	\end{tikzcd}
	\end{equation*}
	(For $R\rightarrow S \xrightarrow{h} S'$, $h_\ast$ is the natural morphism from $\Omega_{S/R}\otimes_S S'$ to $\Omega_{S'/R}$, i.e. $h_*(\mathrm{d}x) = \mathrm{d}h(x)$. )\\
	So by the universal property of colimit, we have a unique morphism $\phi: \text{colim} (\Omega_{A_i/A}\otimes_{A_i} B)\rightarrow\Omega_{B/A}$ induced by all $g_{i,*}$.\\
	On the other hand, consider the natural map $\delta: B = \text{colim}(A_i) \xrightarrow{\mathrm{d}} \text{colim} (\Omega_{A_i/A}\otimes_{A_i} B)$. For all $b_1,b_2\in B$, there are $a_1,a_2\in A_i$ for some $i$, such that $g_i(a_t)=b_t$ for $t = 1,2$. Then $$\delta(b_1b_2)=[\mathrm{d}(a_1a_2)]=[a_2\mathrm{d}(a_1)]+[a_1\mathrm{d}(a_2)]$$ Hence $\delta$ is a $B$-derivation. From this, one gets a morphism $\varphi$ from $\Omega_{B/A}$ to $\text{colim} (\Omega_{A_n/A}\otimes_{A_n} B)$ such that $\delta = \varphi\mathrm{d}$. By the universal property of K\"{a}hler differential, $\phi\circ\varphi$ is the identity. Moreover, for all $\omega\in\text{colim} (\Omega_{A_i/A}\otimes_{A_i} B)$ can be expressed by an element $ w_i\in\Omega_{A_i/A}\otimes_{A_i} B$ which is generated by $\mathrm{d}a_i\in A_i$. Hence $\delta$ is a surjective as $B$-Module. It follows that $\varphi$ is surjective and therefore $\phi$ is an isomorphism.\\
\end{proof}
	\Lemma \label{Lem.Vn}
	Denote $V_n := \mathbb{Z}_p[p^{1/p^n}]$, and $\alpha_n: V_n\hookrightarrow V_{n+1}$, $p^{1/n} \mapsto (p^{1/n+1})^p$ the natural injection. Then $\Omega_{V_n/\mathbb{Z}_p} \cong V_n/(p^n\cdot p^{1-1/p^n})\mathrm{d}p^{1/p^n}$, and $\alpha_{n,\ast}(a\mathrm{d}p^{1/p^n})=p\cdot p^{(p-1)/p^{n+1}}\alpha_n(a)\mathrm{d}p^{1/p^{n+1}}$ for all $a\in V_n$.
	\begin{proof}
	The morphisms $\mathbb{Z}_p\rightarrow\mathbb{Z}_p[X]\xrightarrow{X\rightarrow p^{1/p^n}}\mathbb{Z}_p[p^{1/p^n}]$ induce an exact sequence:$$I/I^2\xrightarrow{\delta}\Omega_{\mathbb{Z}_p[X]/\mathbb{Z}_p}\otimes_{\mathbb{Z}_p[X]}\mathbb{Z}_p[p^{1/p^n}]\rightarrow\Omega_{\mathbb{Z}_p[p^{1/p^n}]/\mathbb{Z}_p}\rightarrow0$$ where $I = (X^{p^n}-p)$. As $\Omega_{\mathbb{Z}_p[X]/\mathbb{Z}_p}\otimes_{\mathbb{Z}_p}\mathbb{Z}_p[p^{1/p^n}] \cong \mathbb{Z}_p[p^{1/p^n}]\mathrm{d}X$, and $\delta(X^{p^n}-p)=p^np^{(p^n-1)/p^n}\mathrm{d}X$, we have $\Omega_{\mathbb{Z}_p[p^{1/p^n}]/\mathbb{Z}_p} \cong \mathbb{Z}_p[p^{1/p^n}]/(p^n\cdot p^{1-1/p^n})\mathrm{d}p^{1/p^n}$. From $\alpha_n(p^{1/p^n}) = (p^{1/p^{n+1}}) ^ p$, one gets  $$\alpha_{n,\ast}(a\mathrm{d}p^{1/p^n}) = \alpha_n(a)\mathrm{d}(p^{1/p^{n+1}})^p = p\cdot p^{(p-1)/p^{n+1}}\alpha_n(a)\mathrm{d}p^{1/p^{n+1}}.$$  
\end{proof}

	\Prop\label{Prop.Vinf}
	All $\alpha_n$ above form a direct system of rings, and we denote the direct limit as $V_{\infty} = \varinjlim V_n$. The following morphism between direct systems:
	\begin{equation*}
	\begin{tikzcd}
	\cdots \arrow {r}& \mathbb{Z}/p^n\mathbb{Z}\otimes V_{\infty} \arrow {d}{\nu_n}\arrow {r}{(\times p) \otimes \id}& \mathbb{Z}/p^{n+1}\mathbb{Z}\otimes V_{\infty} \arrow {d}{\nu_{n+1}}\arrow {r}&\cdots\\
	\cdots \arrow {r}& \Omega_{V_n/\mathbb{Z}_p}\otimes_{V_n} V_{\infty} \arrow {r}{\alpha_{n,\ast}\otimes \id }& \Omega_{V_{n+1}/\mathbb{Z}_p}\otimes_{V_{n+1}} V_{\infty} \arrow {r}&\cdots
	\end{tikzcd}
	\end{equation*}
	where $\nu_n(1\otimes 1)=p^{1-1/p^n}\mathrm{d}p^{1/p^n}\otimes 1$, induces an isomorphism between their direct limits
	$$\nu: \mathbb{Q}_p/\mathbb{Z}_p\otimes V_{\infty} \xrightarrow{\sim} \Omega_{V_{\infty}/\mathbb{Z}_p},\ \frac{1}{p^n}\otimes 1 \mapsto p^{1-\frac{1}{p^n}}\mathrm{d}p^{\frac{1}{p^n}}.$$

	\begin{proof}
		One has
		\begin{align*}
		(\alpha_{n,*}\otimes V_\infty)\circ\nu_n(1\otimes 1) & = (\alpha_{n,*}\otimes V_\infty)(p^{1-1/p^n}\mathrm{d}p^{1/p^n}\otimes 1)\\ 
		&= p\cdot p^{1-1/p^{n+1}}\mathrm{d}p^{1/p^{n+1}}\otimes 1\\
		&= p\cdot\nu_{n+1}(1\otimes 1)\\ &=\nu_{n+1}\circ ((\times p)\otimes\id)(1\otimes 1).
		\end{align*}
		Hence the diagram in the Proposition commutes, which means it is indeed a morphism between direct systems.\\
		From the two Lemmas above, we know the direct limit of the second line is isomorphic to $\Omega_{V_\infty/\mathbb{Z}_p}$. As for the direct limit of the first line, consider the morphism $s_n: \mathbb{Z}/p^n\mathbb{Z} \rightarrow \mathbb{Q}_p/\mathbb{Z}_p, 1 \mapsto 1/p^n$. It follows that  $s_n \otimes \id$ is injective as $s_n$ is injective and $V_\infty$ is a flat $\mathbb{Z}_p$-module (because $V_\infty$ is the direct limit of free modules $\mathbb{Z}_p[p^{1/p^n}]$). Hence $\varinjlim s_n: \varinjlim \mathbb{Z}/p^n\mathbb{Z}\otimes V_\infty \rightarrow \mathbb{Q}_p/\mathbb{Z}_p\otimes V_\infty$ is injective. Combined with $\cup \mathrm{Im}(s_n) = \mathbb{Q}_p/\mathbb{Z}_p$, one gets the direct limit of first line is $\mathbb{Q}_p/\mathbb{Z}_p\otimes V_\infty$.\\
		As  $\Omega_{V_n/\mathbb{Z}_p}\otimes_{V_n} V_{\infty}$ is generated by $\mathrm{d}p^{1/p^n}$ as a $V_\infty$-module, $\Omega_{V_\infty/\mathbb{Z}_p}$ is generated by $(\mathrm{d}p^{1/p^n})_{n\in\mathbb{N}}$.
		For all $n\in \mathbb{N}$ 
		\begin{align*}
        (\alpha_{n,\ast}\otimes\id) (\mathrm{d}p^{1/p^n}\otimes 1) &= p\cdot p^{(p-1)/p^{n+1}}\mathrm{d}p^{1/p^{n+1}} \otimes 1\\
         &= p^{1-1/p^{n+1}}\mathrm{d}p^{1/p^{n+1}}\otimes p^{1/p^n}\\
          &= \nu_{n+1}(1\otimes p^{1/p^n})
		\end{align*}
		which implies $\nu$ is surjective.\\
		Each $\nu_n$ is injective, which also comes from the flatness of $V_\infty$, by the exactness of direct limit, one gets that $\nu$ is injective. Finally, $\nu(1/p^n\otimes 1) = \nu_n(1\otimes 1) = p^{1-1/p^n}\mathrm{d}p^{1/p^n}$.
	\end{proof}
    \Lemma
    Let $k$ be a field with characteristic $p$, and assume $[k:k^p]<\infty$, here $k^p$ denotes the image of Frobenius embedding. Then $[k:k^p] =p^r$ for some $r$ ($r$ is called the $p$-rank of $k$), and for all finite extensions  $l/k$, the $p$-rank of $l$ is $r$ and $\Omega_{l/k}$ can be generated by at most $r$ elements.
    \begin{proof}
    If $k$ is a finite extension of $k^p$, one can choose $x_1,\dots,x_n\in k$ for some $n$ to generate $k$ as a $k^p-$algebra. As for all $x\in k$, $x^p\in k^p$, one gets $[k_{i+1}/k_{i}] = p$ or 1, where $k_i := k^p(x_1,\dots,x_i)$. Hence $[k:k^p]=p^r$ for some $r$. Then 
    consider the exact sequence:
    $$\Omega_{k(l^p)/k}\otimes l \rightarrow \Omega_{l/k} \rightarrow \Omega_{l/k(l^p)}\rightarrow 0$$
    It follows that $\Omega_{l/k} \cong \Omega_{l/k(l^p)}$ as the first arrow is zero map($\mathrm{d}x^p = px^{p-1}\mathrm{d}x =0$). For the reason that  $[l:l^p][l:k]=[l:l^p][l^p:k^p]=[l:k^p]=[l:k][k:k^p]$, one gets $[l:l^p]=[k:k^p]=r$. So $[l:k(l^p)]=p^s | p^r$, which means $l$ can be generated by at most $r$ elements as a $k(l^p)$-algebra. Hence $\Omega _{l/k}$ can be generated by at most $r$ elements.
    \end{proof}
\Prop \label{Prop.KahlerofRingofInteger}
Let $K$ be a local field with  residue field $k$ of characteristic $p$ whose $p$-rank is $r$, $L$ be a finite extension of $K$ with residue field $l$. Then $\Omega_{\mathcal{O}_L/\mathcal{O}_K}$ is generated by at most $r+1$ elements. Moreover, if $L/K$ is separable, $\Omega_{\mathcal{O}_L/\mathcal{O}_K}$ can be expressed as the form $\bigoplus_{i=1,\dots,n} \mathcal{O}_L/(\pi^{e_i})$ for some $n \leq r+1$ and $e_i\in\mathbb{N}$. Here $\pi$ is a uniformizer of $\mathcal{O}_L$.
\proof
Consider the following two exact sequences:\\
$$0=\Omega_{k/\mathcal{O}_K}\otimes l\rightarrow \Omega_{l/\mathcal{O}_K} \rightarrow \Omega_{l/k} \rightarrow 0$$
$$(\pi)/(\pi)^2\rightarrow \Omega_{\mathcal{O}_L/\mathcal{O}_K}\otimes l\rightarrow \Omega_{l/\mathcal{O}_K} \rightarrow 0 $$
One gets $n= \mathrm{dim}(\Omega_{\mathcal{O}_L/\mathcal{O}_K}\otimes l) \leq \mathrm{dim}(\Omega_{l/k})+1 \leq r+1$. As $\mathcal{O}_L$ is a discrete valuation ring, in particular is a prime ideal domain, by the structure theory of finitely generated modules over prime ideal domain, $\Omega_{\mathcal{O}_L/\mathcal{O}_K}= \bigoplus_{i=1,\dots,n} \mathcal{O}_L/(\pi^{e_i})\bigoplus \mathcal{O}_L^m$ for some $n,m,e_i\in\mathbb{N}$ with $n+m \leq r+1$.\\
If $L/K$ is separable, it is easy to see that $\Omega_{\mathcal{O}_L/\mathcal{O}_K}$ has no free part as for all $a\in\mathcal{O}_L$ with minimal polynomial $f(X)\in\mathcal{O}_K[X]$, $f'(a)\mathrm{d}a = \mathrm{d}f(a) = 0$. Hence $m = 0$.
\Lemma \label{Lem.RingofIntegerStructure}
Let $v$ be a discrete valuation of a field $L$, and $K$ is a subfield of $L$ such that $[L:K]<\infty$. We denote $l$ (resp. $k$) the residue field of $L$ (resp. $K$) w.r.t. $v$, then:
\item 1.If $[L:K]=[l:k]$ and $l = k(\bar{x})$ for some $x\in \mathcal{O}_L$, then $\mathcal{O}_L =\mathcal{O}_K[x]$.
\item 2.If $[L:K] = [v(\mathcal{O}_L):v(\mathcal{O}_K)]$ (i.e. $m_K \mathcal{O}_L=m_L^{[L:K]}$), then $\mathcal{O}_L = \mathcal{O}_K[\pi]$, where $m_L = (\pi)$.
\begin{proof}
	\item For case 1: Let $\pi$ be the uniformizer of $L$ (also for $K$ as $e(L|K) =1$). Then $$\mathcal{O}_L/(\pi) = l = k(\bar{x}) = \mathcal{O}_K[x]/(\pi).$$
	It follows from Nakayama Lemma.
	\item For case 2: the degree of $\pi$  is smaller or equal to $n = [L:K]$. But it should also be greater or equal to $n$ as for $a_i\in K$ which are not all 0, one has
	$$v(a_0+a_1 \pi+\dots+a_{n-1}\pi^{n-1}) = v(a_i\pi^i)$$
    where $i$ is the smallest subscript such that $a_i = \min\limits_{0\leq j<n}\{a_j\}$. From the equation about, we also know that if $a_0+a_1 \pi+\dots+a_{n-1}\pi^{n-1}$ is an integer, each $a_i$ should in $\mathcal{O}_K$. Hence $\mathcal{O}_L = \mathcal{O}_K[\pi]$.
\end{proof} 
\Prop\label{Prop.SESofKahler}
Let $U\subset V\subset W$ be a sequence of finite extensions of complete discrete valuation rings with charateristic 0, then the sequence
$$0\rightarrow \Omega_{V/U}\otimes W\rightarrow \Omega_{W/U} \rightarrow \Omega_{W/V}\rightarrow 0$$
is exact.
\begin{proof}
	By Cor 8.11, $H_1(\mathbb{L}_{W/V}) =0$. Hence the sequence is exact.
\end{proof}
\Prop\label{Prop.PerfectoidDescent} Denote by $\bar{V}$ the valuation ring of $\bar{\mathbb{Q}}_p$, then the natural morphism $$\Omega_{V_\infty/\mathbb{Z}_p}\otimes_{V_\infty}\bar{V} \rightarrow \Omega_{\bar{V}/\mathbb{Z}_p}$$
is an isomorphism.
\begin{proof}
	Denote $E_n := \mathbb{Q}_p(p^{1/p^n})$, and let $E_\infty$ be their direct limit. For any finite algebraic field extension $ F_\infty = E_\infty(x)$ of $E_\infty$, denote $F_n:=E_n(x) $. For all $n \in\mathbb{N}\cup\{\infty\}$, denote $V_n$ (resp. $W_n$) be the valuation ring of $E_m$ (resp. $F_m$). Notice that $\mathbb{Z}_p[p^{1/p^n}]$ is the valuation ring of $E_n$ by the case 2 of Lemma \ref{Lem.RingofIntegerStructure}, then the notations here are consistent with that in Lemma \ref{Lem.Vn} and Proposition \ref{Prop.Vinf}.\\
	Consider the following direct system of short exact sequences (Proposition 1.7, using $W_\infty$ is flat over $W_n$)
	\begin{equation*}
	\begin{tikzcd}
	0 \arrow {r}& \Omega_{V_n/V_0}\otimes W_\infty \arrow {d}\arrow {r}& \Omega_{W_n/V_0}\otimes W_\infty \arrow {d}\arrow {r}&\Omega_{W_n/V_n}\otimes W_\infty \arrow{r} \arrow{d}{\Phi_n} & 0\\
	0 \arrow {r}& \Omega_{V_{n+1}/V_0}\otimes W_\infty \arrow {r}& \Omega_{W_{n+1}/V_0}\otimes W_\infty \arrow {r}&\Omega_{W_{n+1}/V_{n+1}}\otimes W_\infty \arrow{r}  & 0
	\end{tikzcd}
	\end{equation*}
	Take the direct limit, and by Lemma \ref{Lem.DirectLimitofKahler}, the sequence comes to be 
	\begin{equation}
	0 \rightarrow \Omega_{V_\infty/V_0}\otimes W_\infty\rightarrow \Omega_{W_\infty/V_0} \rightarrow \varinjlim(\Omega_{W_n/V_n}\otimes W_\infty) \rightarrow 0 
	\end{equation}
	
	We are going to show that $\varinjlim(\Omega_{W_n/V_n}\otimes W_\infty) = 0$. Consider the following diagram
	\begin{equation*}
	\begin{tikzcd}[column sep=small]
	0\arrow {rd}& & & & 0\\
	& \Omega_{V_{n+1}/V_n} \otimes W_{n+1} \arrow {rd}{\alpha} & &\Omega_{W_{n+1}/W_n} \otimes W_{n+1} \arrow {ru}& \\
	& & \Omega_{W_{n+1}/V_n} \otimes W_{n+1} \arrow {rd}{\beta}\arrow {ru}& & \\
	& \Omega_{W_n/V_n} \otimes W_{n+1}\arrow {ru}{\gamma}\arrow {rr}{\phi_n}&  & \Omega_{W_{n+1}/V_{n+1}}  \arrow{rd}& \\
	0\arrow {ru}& & & & 0
	\end{tikzcd}
	\end{equation*}
	according to Proposition \ref{Prop.SESofKahler}, the two slash lines are short exact sequences, and the triangle in the diagram commutes, and $\phi_n \otimes id_{W_\infty} = \Phi_n$. As the residue fields of $V_n$ and $W_n$ above are finite fields, whose $p$-rank is 0. Hence all the modules above are generated by one element by Proposition \ref{Prop.KahlerofRingofInteger}. Suppose $ \Omega_{W_n/V_n}  = W_n/(p^{\delta_n}W_n) $, here we denote by $p^{\delta_n}$ some element with normalized valuation $\delta_n$(i.e. $v(p)=1$) for convenience. By direct computation, it is easy to show that:     $$\Omega_{V_{n+1}/V_n} \otimes W_{n+1} = W_{n+1}/(p^{1+\varepsilon_n}W_{n+1}), \varepsilon_n = (p-1)/p^n $$ 
	$$ \Omega_{W_{n+1}/V_n} \otimes W_{n+1} = W_{n+1}/(p^{\delta_{n+1}+1+\varepsilon_n}W_{n+1})$$
	$$\alpha =  p^{\delta_{n+1}}\times\_\_{\text{  (up to unit)}}$$
	$$\gamma =  p^{ \delta_{n+1}-\delta_{n}+1+\varepsilon_n}\times\_\_\text{  (up to unit})$$
	and $\beta$ is the canonical quotient map, so let $\omega_n\in\Omega_{W_n/V_n}$ be the basis (i.e. $\Omega_{W_n/V_n}= W_n \omega_n/(p^{\delta_n}\omega_n)$). Then up to unit, $\Phi_n(\omega_n) = p^{\delta_{n+1}-\delta_{n}+1+\varepsilon_n}\omega_{n+1}$. Then $\Phi_{n+k-1}\dots\Phi_n(\omega_n) = p^{\delta_{n+k}-\delta_n+k+\varepsilon}\omega_{n+k}$ is zero in $\Omega_{W_{n+k}/V_{n+k}}\otimes W_\infty$ for any natural number $ k>\delta_n$, where $\varepsilon = \varepsilon_n + \dots+ \varepsilon_{n+k-1}$. This implies that  $\varinjlim(\Omega_{W_n/V_n}\otimes W_\infty) = 0$. Hence when $F_\infty$ runs through all finite algebraic extensions  of $E_\infty = \mathbb{Q}_p(p^{1/p^\infty})$, one gets the isomorphism from taking the filtered colimit of sequence (1) and using Lemma \ref{Lem.DirectLimitofKahler} again:  $$\Omega_{V_\infty/\mathbb{Z}_p}\otimes_{V_\infty} \bar{V} \cong \Omega_{\bar{V}/\mathbb{Z}_p}.$$ 
\end{proof}
    \Cor
    Let $\bar{V}$ be as above, then the map
    $$\lambda: (\mathbb{Q}_p/\mathbb{Z}_p)\otimes\bar{V} \rightarrow\Omega_{\bar{V}/\mathbb{Z}_p}, \frac{1}{p^n}\otimes \alpha \mapsto \alpha\cdot p^{1-\frac{1}{p^n}}\mathrm{d} p^{\frac{1}{p^n}} $$ is an isomorphism. 
    \begin{proof}
    	According to Proposition \ref{Prop.Vinf}. $1/p^n \mapsto p^{1-1/p^n}\mathrm{d}p^{1/p^n}$ gives an isomorphism from $(\mathbb{Q}_p/\mathbb{Z}_p)\otimes V_\infty$ to $ \Omega_{V_\infty/\mathbb{Z}_p}$. Combined with Proposition \ref{Prop.PerfectoidDescent} , one gets the isomorphism by $-\otimes_{V_\infty}\bar{V}$.
    \end{proof}
    \section{Analogue for $\mathcal{O}_{K_0}$ }
    Now we extend the $p$-adic valuation to $\mathbb{Q}_p(T)$ with discrete valuation ring $\mathbb{Z}_p[T]_{(p)}$ (i.e. the localization of $\mathbb{Z}_p[T]$ at prime ideal $(p)$). Denote by $K_0$ the completion of $\mathbb{Q}_p(T)$ and $\bar{K}_0$ the algebraic closure of $K_0$. In this section, our goal is to compute the structure of $\Omega_{\mathcal{O}_{\bar{K}_0}/\mathcal{O}_{K_0}}$. Our tragedy is similar to the last section. Let $K_{\infty}$ be a perfectoid extension of $K_0$. We compute $\Omega_{\mathcal{O}_{k_{\infty}}/\mathcal{O}_{k_{0}}}$ and then prove $$\Omega_{\mathcal{O}_{k_{\infty}}/\mathcal{O}_{k_{0}}} \otimes \mathcal{O}_{\bar{K}_0} \cong \Omega_{\mathcal{O}_{\bar{K}_0}/\mathcal{O}_{k_{0}}}.$$ Some repeated computations will be omitted, which have been occurred in last section. 
    \Prop\label{Prop.pRankofK0}
    The $p$-rank of the residue field of $K_0$ is 1.
    \begin{proof}
    First, it is easy to see that the residue field $k_0$ of $\mathcal{O}_{K_0}$ is $\mathbb{F}_p(T)$. As $\mathbb{F}_p\subset (k_0)^p$, one gets $$\mathbb{F}_p(T) \subset (k_0)^p(T)\subset \mathbb{F}_p(T)$$ Hence the p-rank of $k_0$ is at most 1 (as $k_0 = k_0^p(T)$ and the degree of $T$ is at most p). On the other hand, $T\notin(k_0)^p$ as for all $f,g \in \mathbb{F}_p[T]$, $p|\mathrm{deg}(f^p/g^p) = p(\mathrm{deg} (f)-\mathrm{deg} (g))$. It follows that $k_0 \neq (k_0)^p$ and the p-rank of $k_0$ is not zero. Hence it is equal to 1.
    \end{proof}
    \Prop\label{Prop.KahlerofUinf}
    Let $U_n := \mathcal{O}_{K_0}[p^{1/p^n},T^{1/p^n}]$ and $U_\infty := \varinjlim U_n$. Then the map $$\mu:(\mathbb{Q}_p/\mathbb{Z}_p)^{\oplus 2}\otimes U_\infty\rightarrow \Omega_{U_\infty/\mathcal{O}_{K_0}},\ (\frac{1}{p^m}\otimes a,\frac{1}{p^n}\otimes b) \mapsto a\cdot p^{1-1/p^m}\mathrm{d}p^{1/p^m}+b\cdot T^{1-1/p^n}\mathrm{d}T^{1/p^n}$$
    is an isomorphism of $U_\infty$-modules.
    \begin{proof}
    	The morphisms $\mathcal{O}_{K_0}\rightarrow \mathcal{O}_{K_0}[X,Y] \xrightarrow{X\mapsto p^{1/p^n},Y\mapsto T^{1/p^n}}U_n$ induce an exact sequence:
    	$$ I/I^2\xrightarrow{\delta}\Omega_{\mathcal{O}_{K_0}[X,Y]/\mathcal{O}_{K_0}}\rightarrow\Omega_{U_n/\mathcal{O}_{K_0}}\rightarrow 0$$
    	where $I = (X^{p^n}-p,Y^{p^n}-T)$. Hence by a computation similar to Lemma 1.2, one gets $$\Omega_{U_n/\mathcal{O}_{K_0}} = U_n/(p^n\cdot p^{1-1/p^n})\mathrm{d}p^{1/p^n}\oplus U_n/(p^n)\mathrm{d}T^{1/p^n}$$
    	Consider the following morphism between direct systems:
    	\begin{equation*}
    	\begin{tikzcd}
    	\cdots \arrow {r}& (\mathbb{Z}/p^n\mathbb{Z})^{\oplus 2}\otimes U_{\infty} \arrow {d}{\mu_n}\arrow {r}{(\times p) \otimes id}& (\mathbb{Z}/p^{n+1}\mathbb{Z})^{\oplus 2}\otimes U_{\infty} \arrow {d}{\mu_{n+1}}\arrow {r}&\cdots\\
    	\cdots \arrow {r}& \Omega_{U_n/\mathcal{O}_{K_0}}\otimes_{U_n} U_{\infty} \arrow {r}{\beta_{n,\ast}\otimes id }& \Omega_{U_{n+1}/\mathcal{O}_{K_0}}\otimes_{U_{n+1}} U_{\infty} \arrow {r}&\cdots
    	\end{tikzcd}
    	\end{equation*}
    	where $\mu_n((a,b)) = p^{1-1/p^n}\mathrm{d}p^{1/p^n}\otimes a+ T^{1-p^n}\mathrm{d}T^{1/p^n}\otimes b$, and $\beta_{n,*}$ is induced by the natural injection $\beta_n: U_n \hookrightarrow U_{n+1}$. We also need to check the diagram commutes:
    	\begin{align*}
    	(\beta_{n,*}\otimes id)\cdot \mu_n((a,b)) & = (\beta_{n,*}\otimes id)(p^{1-1/p^n}\mathrm{d}p^{1/p^n}\otimes a+ T^{1-p^n}\mathrm{d}T^{1/p^n}\otimes b)\\
    	&=
    	p\cdot p^{1-1/p^{n+1}}\mathrm{d}p^{1/p^{n+1}}\otimes a+ p\cdot T^{1-p^{n+1}}\mathrm{d}T^{1/p^{n+1}}\otimes b\\
    	&=
    	\mu_{n+1}(p(a,b)))
    	\end{align*}
    As we have proved that $\varinjlim \mathbb{Z}/p^n\mathbb{Z}=\mathbb{Q}_p/\mathbb{Z}_p$ in Proposition \ref{Prop.Vinf}, the direct limit of the first line is $(\mathbb{Q}_p/\mathbb{Z}_p)^{\oplus 2}\otimes U_\infty$. And the direct limit of second line is $\Omega_{U_\infty/\mathcal{O}_{K_0}}$. And $\mu$ is induced by the maps $(1/p^n,0) \mapsto p^{1-1/p^n}\mathrm{d}p^{1/p^n}$ and $(0,1/p^n) \mapsto T^{1-1/p^n}\mathrm{d}T^{1/p^n}$. \\
    As each $\mu_n$ is injective, $\mu$ is also injective. In addition, $$\mathrm{d}T^{1/p^n}\otimes 1 = \mu_n((0,1)\otimes T^{1/p^n-1})$$ and $$(\beta_{n,*}\otimes id)(\mathrm{d}p^{1/p^n}\otimes 1) = p\cdot p^{(p-1)/p^{n+1}}\mathrm{d}p^{1/p^{n+1}}=\mu_{n+1}((p^{1/n},0))$$ Hence it is also surjective.
    \end{proof}
    \Prop\label{Prop.KalherofUbar}
    $U_\infty$ is the valuation ring of $K_0(p^{1/p^\infty},T^{1/p^\infty}):= K_0(p^{1/p},p^{1/p^2},\cdots;T^{1/p},T^{1/p^2},\cdots)$, and let $\bar{U}$ be the  valuation ring of $\bar{K}_0$. The natural morphsim:
    $$\Omega_{U_\infty/\mathcal{O}_{K_0}}\otimes\bar{U} \rightarrow \Omega_{\bar{U}/\mathcal{O}_{K_0}}$$
    is an isomorphism. 
    \begin{proof}
    Denote $E_n := K_0(p^{1/p^n},T^{1/p^n})$.
    According to Lemma \ref{Lem.RingofIntegerStructure}, $\mathcal{O}_{K_0(T^{1/p^n})} =\mathcal{O}_{K_0}[T^{1/p^n}] $(case 1) and $\mathcal{O}_{E_n} = \mathcal{O}_{K_0(T^{1/p^n})}[p^{1/p^n}]$(case 2). Hence $U_n :=\mathcal{O}_{K_0}[P^{1/p^n},T^{1/p^n}] = \mathcal{O}_{E_n}$. It follows that $U_\infty = \mathcal{O}_{E_\infty}$.
    Let $F_\infty := E_\infty(x)$ be some finite extension of $E_\infty$, denote $F_n := E_n(x)$, and denote $W_n:= \mathcal{O}_{F_n}$, $W_\infty:= \mathcal{O}_{F_\infty}$.\\
    Take the direct limit of the following exact sequence:
    \begin{equation*}
0\rightarrow \Omega_{U_n/\mathcal{O}_{K_0}}\otimes W_\infty \rightarrow\Omega_{W_n/\mathcal{O}_{K_0}}\otimes W_\infty\rightarrow\Omega_{W_n/U_n}\otimes W_\infty\rightarrow0
    \end{equation*}
    (The exactness comes from Proposition 1.7 and $W_\infty$ is a flat $W_n$-module). One will get the short exact sequence:
    \begin{equation}
 0\rightarrow \Omega_{U_\infty/\mathcal{O}_{K_0}}\otimes W_\infty \rightarrow\Omega_{W_\infty/\mathcal{O}_{K_0}}\rightarrow\varinjlim_n(\Omega_{W_n/U_n}\otimes W_\infty)\rightarrow0 
    \end{equation}
    Now we are going to prove that $\varinjlim(\Omega_{W_n/U_n}\otimes W_\infty) = 0$. Consider the following short exact sequence:
    $$0\rightarrow\Omega_{U_{n+1}/U_n}\otimes_{U_{n+1}} W_{n+1}\cong W_{n+1}/(p^{1+(p-1)/p^{n+1}})\oplus W_{n+1}/(p)\xrightarrow{f_n}\Omega_{W_{n+1}/U_n}\xrightarrow{g_n}\Omega_{W_{n+1}/U_{n+1}}\rightarrow0$$
    By the snake Lemma ($\times p$ to above sequence), $\Omega_{U_{n+1}/U_n}\otimes_{U_{n+1}} W_{n+1}[p] = (W_{n+1}/(p))^{\oplus 2}\hookrightarrow \Omega_{W_{n+1}/U_n}[p]$. As $W_{n+1}/U_n$, $U_n/\mathcal{O}_{K_0}$ are finite extension, by Proposition \ref{Prop.KahlerofRingofInteger} and Proposition \ref{Prop.pRankofK0}, $\Omega_{W_{n+1}/U_n}$ is generated by $2$ elements. Consider the following exact sequence:
    $$0\rightarrow W_{n+1}^{\oplus 2}\xrightarrow{(\_\times a,\_\times b)} W_{n+1}^{\oplus 2}\rightarrow\Omega_{W_{n+1}/U_n}\rightarrow0.$$
    $\times p$ to the sequence and use the snake Lemma again, one gets an injection: $\Omega_{W_{n+1}/U_n}[p]\hookrightarrow (W_{n+1}/(p))^{\oplus 2}$. Hence one gets an injection :$(W_{n+1}/(p))^{\oplus 2}\hookrightarrow (W_{n+1}/(p))^{\oplus 2}$, which is easy to know that is bijection. It follows that $\mathrm{ker}(g_n) = \mathrm{Im} (f_n) \supseteq \Omega_{W_{n+1}/U_n}[p]$.\\
    Notice that the natural morphism $\Omega_{W_n/U_n}\otimes W_\infty \xrightarrow{h_n} \Omega_{W_{n+1}/U_{n+1}}\otimes W_\infty$ comes from the natural morphism $\Omega_{W_n/U_n}\otimes W_\infty \rightarrow \Omega_{W_{n+1}/U_n}\otimes W_\infty$ composed with $g_n$. Then $\Omega_{W_n/U_n}\otimes W_\infty[p] \subseteq \mathrm{ker}(h_n)$. Hence for $k$ large enough, $\Omega_{W_n/U_n}\otimes W_\infty = \Omega_{W_n/U_n}\otimes W_\infty[p^k] \subseteq \mathrm{ker}(h_n\circ\dots\circ h_{n+k-1})$. It follows that $\varinjlim(\Omega_{W_n/U_n}\otimes W_\infty) = 0$.
    Then when $F_\infty$ runs through all finite extension of $E_\infty$ and take the filtered colimit of short exact sequence (2), one gets: $$\Omega_{U_\infty/\mathcal{O}_{K_0}}\otimes\bar{U} \xrightarrow{\sim} \Omega_{\bar{U}/\mathcal{O}_{K_0}}.$$
    \end{proof}
    
    \begin{Cor}\label{Cor.KahlerDiff}
    	Let $\bar{U}$ be as above, then the map:
    	\begin{align*}
    		\lambda:& (\mathbb{Q}_p/\mathbb{Z}_p)^{\oplus 2}\otimes \bar{U}\rightarrow \Omega_{\bar{U}/\mathcal{O}_{K_0}}\\
    		&\frac{1}{p^m}\otimes \alpha +\frac{1}{p^n}\otimes \beta \mapsto \alpha p^{1-1/p^m}\mathrm{d}p^{1/p^m}+\beta T^{1-1/p^n}\mathrm{d}T^{1/p^n} 
    	\end{align*}
    	is an isomorphism.
    	
    \end{Cor}
    \begin{proof}
    	This is a direct corollary from Proposition \ref{Prop.KahlerofUinf} and Proposition \ref{Prop.KalherofUbar}.
    \end{proof}
    \section{The problem of coordinate transformation}
    \Prop
    Let $S$ be an element in $\mathcal{O}_{K_0}$ such that $\bar{S} \in \mathbb{F}_p(T) \backslash (\mathbb{F}_p(T))^p$, then the map
    \begin{align*}
    \lambda_S:& (\mathbb{Q}_p/\mathbb{Z}_p)^{\oplus 2}\otimes \bar{U}\rightarrow \Omega_{\bar{U}/\mathcal{O}_{K_0}}\\
    &\frac{1}{p^m}\otimes \alpha +\frac{1}{p^n}\otimes \beta \mapsto \alpha p^{1-1/p^m}\mathrm{d}p^{1/p^m}+\beta S^{1-1/p^n}\mathrm{d}S^{1/p^n} 
    \end{align*}
    is an isomorphism.
    \begin{proof}
    	For the same reason as we have done, if denote $U_{S,n} := \mathcal{O}_{K_0(p^{1/p^n},S^{1/p^n})} =\mathcal{O}_{K_0}[p^{1/p^n},S^{1/p^n}]$ and $U_{S,\infty}$ as their direct limit, one gets an isomorphism 
    	$$(\mathbb{Q}_p/\mathbb{Z}_p)^{\oplus 2}\otimes U_{S,\infty}\rightarrow \Omega_{U_{S,\infty}/\mathcal{O}_{K_0}}, \frac{1}{p^m}\otimes \alpha +\frac{1}{p^n}\otimes \beta \mapsto \alpha p^{1-1/p^m}\mathrm{d}p^{1/p^m}+\beta S^{1-1/p^n}\mathrm{d}S^{1/p^n} $$
    	And repeats the proof of Proposition \ref{Prop.KalherofUbar}, one gets the isomorphsim.
    \end{proof}

By the Proposition above, $(p^{1-1/p^n}\mathrm{d}p^{1/p^n}, T^{1-1/p^n}\mathrm{d}T^{1/p^n})$ and $(p^{1-1/p^n}\mathrm{d}p^{1/p^n}, S^{1-1/p^n}\mathrm{d}S^{1/p^n})$ both make up a basis of $\Omega_{\bar{U}/\mathcal{O}_{K_0}}[p^n]$ as a $\bar{U}/(p^n)$-module. Hence there must exist some formula correspond to the transformation for these coordinates. More precisely, there exists $a\in \bar{U}/(p^n),b,c\in(\bar{U}/(p^n))^\times$ depending on $n,S$, such that:
$$ap^{1-1/p^n}\mathrm{d}p^{1/p^n}+bT^{1-1/p^n}\mathrm{d}T^{1/p^n}+cS^{1-1/p^n}\mathrm{d}S^{1/p^n}=0$$
The main goal of this article is to find such formula for general $p,n$ in the case of $S+T=1$.
    \section{The formula for the coordinate transformation} \label{Sect.CoeffFormula}
    In this section, we recall the construction of Witt vectors and Fontaine's map. Then we will use these constructions to give the formula for the coordinate transformation.
    \Lemma
    Fix a prime number $p$. For all $n\in \mathbb{N}$, let $W_n(X_0,\dots,X_n)=\sum_{k=0}^n p^kX_k^{p^{n-k}} $. Then for every polynomial function $\Phi\in \mathbb{Z}[X,Y]$, there exist a unique sequence $(\phi_0,\dots,\phi_n,\dots)$ of elements of $\mathbb{Z}[X_0,\dots,X_n,\dots;Y_0,\dots,Y_n,\dots]$, such that 
    $$W_n(\phi_0,\dots,\phi_n)=\Phi(W_n(X_0,\dots,X_n),W_n(Y_0,\dots,Y_n))$$
    \proof
    The uniqueness is obvious ($W_n(\phi_0,\dots,\phi_{n-1},\phi_n) = W_n(\phi_0,\dots,\phi_{n-1},\phi_n') \Rightarrow p^n\phi_n=p^n\phi_n'$). As for the existence, let $\phi_0 = \Phi(X_0,Y_0)$. Then assume $\phi_0,\dots,\phi_n$ has been found, we only need to prove that $$p^{n+1}|\Phi(W_n(X^p_0,\dots,X^p_n),W_n(Y^p_0,\dots,Y^p_n))-W_n(\phi^p_0,\dots,\phi^p_n)$$ as $W_{n+1}(X_0,\dots,X_{n+1})=W_n(X^p_0,\dots,X^p_n) + p^{n+1}X_{n+1}$. Note that 
    \begin{align*}
       &\Phi(W_n(X^p_0,\dots,X^p_n),W_n(Y^p_0,\dots,Y^p_n))-W_n(\phi^p_0,\dots,\phi^p_n)\\ =& \sum_{k=0}^{n}p^k\phi_k^{p^{n-k}}(X_0^p,\dots,Y_k^p)-\sum_{k=0}^{n}p^k\phi_k^{p^{n+1-k}}(X_0,\dots,Y_k)\\
       =&\sum_{k=0}^{n}p^k(\phi_k^{p^{n-k}}(X_0^p,\dots,Y_k^p)-\phi_k^{p^{n+1-k}}(X_0,\dots,Y_k)).
    \end{align*}
    As $p^{n+1-k}|\phi_k^{p^{n-k}}(X_0^p,\dots,Y_k^p)-\phi_k^{p^{n+1-k}}(X_0,\dots,Y_k)$, one gets the conclusion.
    \Prop
    Applying the Lemma above we denote $S_i$ (resp. $P_i$) the polynomials $\phi_i$ associated to $\Phi(X,Y) = X+Y$ (resp. $\Phi(X,Y) = XY)$. Let $R$ be an arbitrary commutative ring and $A$, $B$ be vectors in $R^{\mathbb{N}}$. Then $A+B := (S_0(A,B),S_1(A,B),\dots)$ and $A\cdot B=(P_0(A,B),P_1(A,B),\dots)$ gives $R^{\mathbb{N}}$ a ring structure. The set $R^{\mathbb{N}}$ with this ring structure is called the ring of Witt vectors of R, denoted by $W(R)$.
    \begin{proof}
    	Consider the map:
    	\begin{align*}
    	\chi:(\mathbb{Z}[X_0,X_1,\cdots])^{\mathbb{N}} & \rightarrow(\mathbb{Z}[X_0,X_1,\cdots])^{\mathbb{N}}\\
    	F:=(F_0(X_0,X_1,\cdots),F_1(X_0,X_1,\cdots),\cdots)&\mapsto(W_0(F_0),W_1(F_0,F_1),\cdots)
    	\end{align*}
    From the Lemma above, $\chi$ is injective and $\chi^{-1}(\chi(F)+\chi(G))$(resp. $\chi^{-1}(\chi(F)\cdot\chi(G))$) exists and is equal to $(S_0(F,G),S_1(F,G),\cdots)$ (resp.$(R_0(F,G),R_1(F,G),\cdots)$). As $((\mathbb{Z}[X_0,X_1,\cdots])^{\mathbb{N}},+,\cdot)$ is a ring with identity $(1,1,1,\dots)$ and zero $(0,0,\dots)$, $W(\mathbb{Z}[X_0,X_1,\cdots])$ is a ring with identity $\chi^{-1}(1,1,\cdots)=(1,0,0\dots)$ and zero $\chi^{-1}(0,0,\dots) =(0,0,\dots)$. Consider any ring homomorphism from $\mathbb{Z}[X_0,X_1,\cdots]$ to any countable generated subring $A\subseteq R$, one gets $W(A)$ is a ring, which implies $W(R)$ is a ring. 
    \end{proof}
    \Rmk
    When we say $X\in W(R)$ for some ring $R$, always denote $X_i$ by the $i$th-coordinate of $X$.
    \Prop \label{Prop.FunctorityofWitt}
    $W: Rings \rightarrow Rings$, $R \mapsto W(R)$ is a natural functor from the category of commutative rings to itself.
    \begin{proof}
    	We only need to check that for any two rings $R$ and $R'$, and any morphism $\varphi\in Hom(R,R')$, the map:
    	$$W(\varphi): W(R) \rightarrow W(R'), X := (X_0,X_1,\dots)\mapsto (\varphi(X_1),\varphi(X_2),\dots)$$
    	is a ring homomorphism. Actually, for any $X,Y\in R$ $$(\varphi(X+Y))_i = \varphi (S_i(X,Y)) = S_i(\varphi(X),\varphi(Y)) = (\varphi(X)+\varphi(Y))_i.$$
    Hence $\varphi$ is additive, similarly, it is also multiplicative. Moreover, $ 0=(0,0,0,\dots)$ and $1 = (1,0,0,\dots)$ are preserved by $W(\varphi)$. Hence $W(\varphi)$ is a homomorphism and $W$ is a functor.
\end{proof} 
    \Lemma\label{Lem.4.5}
    Let $R$ be a ring. Suppose $A,B\in W(R)$ such that $A_i=0$ or $B_i=0$ for any $i\in\mathbb{N}$, then 
    $$A+B = (A_0+B_0,A_1+B_1,A_2+B_2,\dots)$$
    \begin{proof}
    	Consider the ring map:
    	\begin{align*}
    	\varphi: \mathbb{Z}[X_0,X_1,X_2,\dots] & \rightarrow R \\
    	X_i &\mapsto
    	\begin{cases}
    	A_i & B_i = 0\\
    	B_i & B_i \neq 0
    	\end{cases}  	
    	\end{align*}
    	Then let $U,V\in W(\mathbb{Z}[X_0,X_1,\dots])$ such that 
    	$$U_i = \begin{cases}
    	X_i & B_i = 0\\
    	0   & B_i \neq 0
    	\end{cases} \text{and } 
    	V_i = \begin{cases}
    	0 & B_i = 0\\
    	X_i   & B_i \neq 0
    	\end{cases}
    	$$
    	Hence $W(\varphi)(U_i) = A_i$, $W(\varphi)(V_i) = B_i$. And it is trivially that $U+V = (X_0,X_1,\dots)$ by checking their associated Witt polynomials. It follows from Proposition \ref{Prop.FunctorityofWitt} that $$A+B = W(\varphi)(U+V) = (A_0+B_0,A_1+B_1,A_2+B_2,\dots).$$
    \end{proof}
    \Prop\label{Prop.4.6} Let $R$ be a ring of characteristic $p$, for all $A = (a_0,a_1,\dots)\in W(R)$, define $V(A) := (0,a_0,a_1,\dots)$ and $F(A) := (a_0^p,a_1^p,\dots)$. Then $V\cdot F = F\cdot V=p$.
    \begin{proof}
    	For $(X_0,X_1,\dots)\in W(\mathbb{Z}[X_0,X_1,\dots])$, denote $p(X_0,X_1,\dots) := (F_0(X_0,X_1,\dots),F_1(X_0,X_1,\dots),\dots)$. Then $p\cdot A =(F_0(A),F_1(A),\dots)$. We only need to prove $F_n(X_0,X_1,\dots) \equiv X_{n-1}^p (\mathrm{mod} p)$.\\
    	First, $F_0 = pX_0$. Assume the assertion is true for $n<m$, then 
    	$F_k^{p^{m-k}} \equiv X_{k-1}^{p^{m-k+1}}(\mathrm{mod}p^{m-k+1})$. Hence
    	\begin{align*}
    	0=&pW_n(X_0,\dots,X_m) - W_n(F_0,\dots,F_m)\\
    	=& \sum_{k=0}^{m}p^{k+1}X_k^{p^{m-k}} -\sum_{k=0}^{m}p^kF_k^{p^{m-k}}\\
    	\equiv& \sum_{k=0}^{m-1}p^{k+1}X_k^{p^{m-k}} - \sum_{k=0}^{m-2}p^{k+1}X_k^{p^{m-k}} +p^mF_m(\mathrm{mod}p^{m+1}).
    	\end{align*}
    	It follows that $F_m \equiv X_{m-1}^p(\mathrm{mod}p)$. By induction, the assertion holds for general $n$.
    \end{proof}
    Now we recall some notations and results about Perfectoid fields and Tilting in \cite{Scholze{2014}}. 
    \Def
    A perfectoid field is a complete topological field $K$, whose topology is induced by a non-archimedean norm $|\cdot|:K \rightarrow \mathbb{R}_{\geq 0}$ with dense image, such that $|p|<1$, and the Frobenius map $\Phi: \mathcal{O}_K/p\rightarrow\mathcal{O}_K/p$ is surjective. Here $\mathcal{O}_K := \{x\in K||x|\leq 1\}$ be the ring of integers.
    \Def
    Let $K$ be a perfectoid field, denote $\mathcal{O}_{k^\flat}:=\varprojlim_{x\rightarrow x^p}\mathcal{O}_K/p$.
    \Prop
    The map:
    $(\ \ )^{\sharp_n}:\mathcal{O}_{K^{\flat}}\rightarrow \mathcal{O}_K,(\bar{x}_0,\bar{x}_1,\dots)\mapsto \lim_{k\geq n} x_k^{p^{k-n}}$ is multiplicative. 
    \begin{proof}
    	As $x_l^{p^{l-k}} \equiv x_k(\mathrm{mod}p)$ for all $l>k\geq n$. Then $x_l^{p^{l-n}} = (x_l^{p^{l-k}})^{p^{k-n}} \equiv x_k^{p^{k-n}}(\mathrm{mod\ }p^{k-n})$. It follows that $(\ \ )^{\sharp_n}$ is well difined. And the multiplicativity is trivial by the following equation 
    	\begin{align*}
    		((\bar{x}_0,\bar{x}_1,\dots)(\bar{y}_0,\bar{y}_1,\dots))^{\sharp_n}&=\lim_{k\geq n} (x_ky_k)^{p^{k-n}}\\ &=\lim_{k\geq n} x_k^{p^{k-n}}\lim_{k\geq n} y_k^{p^{k-n}}\\
    		&=(\bar{x}_0,\bar{x}_1,\dots)^{\sharp_n}(\bar{y}_0,\bar{y}_1,\dots)^{\sharp_n}
    	\end{align*}
    \end{proof}
    \Prop
    Let $K$ be a perfectoid field, then $\mathcal{O}_{K^\flat}$ is a complete valuation ring with valuation $v(x) = v(x^{\sharp_0})$.
    \begin{proof}
    	As $(\ \ )^{\sharp_0}$ is multiplicative, $v(xy)=v(x^{\sharp_0} y^{\sharp_0}) = v(x)+ v(y)$. And for $n$ large enough, one has $$v(x+y) =v((x+y)^{\sharp_0}) = v((x_n+y_n)^{p^n}) \geq min\{v(x_n^{p^n}),v(y_n^{p^n})\}=min\{v(x),v(y)\}.$$
    	If a sequence $(a_0,a_1,\dots)$ in $\mathcal{O}_{K^\flat}$ converges, for all $k$, $\bar{a}_{n,k}$ will be stable as $$(a_{n,k}-a_{m,k})^{p^k}\equiv (a_m-a_n)^{\sharp_0}(\mathrm{mod}p^{k+1})$$ Hence if denote $x_k := \lim \bar{a}_{n,k}$, then $x := (x_0,x_1,\dots)$ will be the limit. It follows that $\mathcal{O}_{K^\flat}$ is complete. 
    \end{proof}
\Prop
Let $K$ be the completion of the algebraic closure of $K_0$. Then $K$ is a perfectoid field.
\begin{proof}
	It is trivial that $K$ is complete and for all $a\in\mathbb{Q}$, $|p^a| = p^{-a}$ implies the image of the norm is dense. Notice that $\mathcal{O}_K /p = \mathcal{O}_{\bar{K}_0}/p$ and $\mathcal{O}_{\bar{K}_0}\xrightarrow{x\rightarrow x^p}\mathcal{O}_{\bar{K}_0}$ is surjective, then the Frobenius map $\Phi$ is surjective.
\end{proof}
    In the rest of this article, we always denote by $K$ the completion of the algebraic closure of $K_0$ and denote $A_{\text{inf}}:= W(\mathcal{O}_{K^\flat})$.
    \Prop[Fontaine's map] The Fontaine's map 
    \begin{align*}
    	\theta: A_{\text{inf}} &\rightarrow \mathcal{O}_K\\
    	 x=(x_0,x_1,\dots) &\mapsto \sum _{n=0}^\infty p^n (x_n)^{\sharp_n}
    \end{align*}
    
    is a ring homomorphism, and $\mathrm{ker}(\theta)$ is generated by $[p^\flat] -p$. Where $[\ ]$ is the Teichm\"{u}ller representative (i.e. $[\ ] : \mathcal{O}_{K^{\flat}} \rightarrow A_{\mathrm{inf}},a \mapsto (a,0,0,\dots)$) and denote $x^\flat := (x,x^{1/p},\dots)$ for $x\in \mathcal{O}_{K}$. 
    
    \begin{proof}
    	 For all $(x_{i,j})_{i,j\in\mathbb{N}} \in \mathcal{O}_K^{\mathbb{N}^2}$ such that $x_{i,j+1}^p \equiv x_{i,j}(\mathrm{mod\ }p )$ denotes an element $x\in A_{\text{inf}}$ in the following way: $x =(x_0,x_1,\dots)$, where $x_i=(\bar{x_{i,0}},\bar{x_{i,1}},\dots)\in \mathcal{O}_{K^\flat}$.\\
    	 For all $x,y\in A_\infty$, let $z =x+y$. If $x =(x_{i,j})_{i,j\in\mathbb{N}}$ and $y=(y_{i,j})_{i,j\in\mathbb{N}}$, then $$z = (z_{i,j} = S_j(x_{i,0},\dots,x_{i,j},y_{i,0},\dots,y_{i,j}))_{i,j\in\mathbb{N}}$$
    	  Hence for any $n$ $$\theta(z) \equiv \sum_{i=0}^{n}p^iz_{i,n}^{p^{n-i}}=\sum_{i=0}^{n}p^i(x_{i,n}^{p^{n-i}}+y_{i,n}^{p^{n-i}})\equiv \theta(x)+\theta(y)(\mathrm{Mod}p^{n+1})$$
    	It follows that the Fontaine's map is additive. Similarly it is multiplicative.\\
    	For the assertion of $\mathrm{ker}(\theta)$, if $\theta(x) =0$, I claim that there exists $ y=(y_0,y_1,\dots)\in A_{\text{inf}}$ such that the first $n-$coordinates of $([p^\flat]-p)(y_0,\dots,y_n,0,0\dots)$ is equals to $(x_0,x_1,\dots)$. As $\theta(x) =0$, $(x_0)^{\sharp_0} \equiv 0(\mathrm{mod}p)$. It follows that $p^\flat|x_0$. So we may choose $y_0 = x_0/p^\flat$. Now we may assume the assertion holds for $n<m$, then consider $z = x-([p^\flat]-p)(y_0,\dots,y_{m-1},0,0\dots)$. It is easy to see that $z_i=0$ for $i<m$, and $\theta(z) = 0$. Hence $p^{m+1}|p^{m}z_m^{\sharp_m}\Rightarrow p|z_m^{\sharp_m}\Rightarrow p^{p^m}|z_m^{\sharp_0}=(z_m^{\sharp_m})^{p^m}\Rightarrow (p^\flat)^{p^m}|z_m$. Then let $y_m = z_m/(p^\flat)^{p^m}$, the first $m$ coordinates of $x-([p^\flat]-p)(y_0,\dots,y_m,0,0\dots)$ is vanishing. By induction, one gets $x =([p^\flat]-p)(y_0,\dots,y_n,\dots)$.
    \end{proof}
    Now as $\theta([S^\flat]+[T^\flat] -1) =0$, there exists $A\in A_{\text{inf}}$ such that $[S^\flat]+[T^\flat] -1 = A([p^\flat]-p)$. Then the formula for the coordinate transformation comes to be 
    $$S^{1-1/p^n}\mathrm{d}S^{1/p^n}+T^{1-1/p^n}\mathrm{d}T^{1/p^n} =\theta(a)p^{1-1/p^n}\mathrm{d}p^{1/p^n}.$$
    
    \section{The "Differential Version" of Fontaine's Map and the Proof}\label{Sect.Fantain'sMap}
    In this section, we will construct a map called the "differential version" of Fontaine's map to prove the formula in last section.
    \Lemma
    If $R$ is a perfect ring of characteristic $p$ (i.e. $R\xrightarrow{x\mapsto x^p}R$ is isomorphism), then $\Omega_{W(R)/\mathbb{Z}_p}$ is uniquely $p$-divisible (i.e. $\times p$ is an isomorphism).
    \begin{proof}
    Consider the map: 
    \begin{align*}
    \mathrm{d}':W(R) &\rightarrow \Omega_{W(R)/\mathbb{Z}_p}\\
    (x_0,x_1,\dots)&\mapsto [x_0^{1/p}]^{p-1}\mathrm{d}[x_0^{1/p}] +\mathrm{d}(x_1^{1/p},x_2^{1/p},\dots)
    \end{align*}
    for all $x,y\in W(R)$ and $z=x+y$, denote $a :=[x_0^{1/p}]$, $b := [y_0^{1/p}]$ and $c := ([(x_0+y_0)^{1/p}] -[x_0^{1/p}] -[y_0^{1/p}])/p$. A general equation holds:
    $$\mathrm{d}\frac{(a+b+pc)^p-a^p-b^p}{p} = (a+b+pc)^{p-1}\mathrm{d}(a+b+pc)-a^{p-1}\mathrm{d}a-b^{p-1}\mathrm{d}b,$$
    as it is true for $\Omega_{\mathbb{Z}[a,b,c]/\mathbb{Z}}$ (you can times $p$ at both side).
     Hence $$\mathrm{d}\frac{[x_0+y_0]-[x_0]-[y_0]}{p}= [(x_0+y_0)^{1/p}]^{p-1}\mathrm{d}[(x_0+y_0)^{1/p}] - [x_0^{1/p}]^{p-1}\mathrm{d}[x_0^{1/p}]- [y_0^{1/p}]^{p-1}\mathrm{d}[y_0^{1/p}]$$
     Notice that (Using the equation that $(x_0,x_1,\dots) = [x_0] + p(x_1^{1/p},x_2^{1/p},\dots)$ according to Lemma \ref{Lem.4.5} and Proposition \ref{Prop.4.6})
     \begin{align*}
     LHS &= \mathrm{d}\frac{x-[x_0]}{p} +\mathrm{d}\frac{y-[y_0]}{p} -\mathrm{d}\frac{x+y-[x_0+y_0]}{p}\\
     &=
     \mathrm{d}(x_1^{1/p},x_2^{1/p},\dots)+\mathrm{d}(y_1^{1/p},y_2^{1/p},\dots)-\mathrm{d}(z_1^{1/p},z_2^{1/p},\dots).
     \end{align*}
It follows that $\mathrm{d}'$ is additive. On other hand
     \begin{align*}
     \mathrm{d}'(xy) &= \mathrm{d}'{([x_0]+(x-[x_0]))([y_0]+(y-[y_0]))}\\&=
     \mathrm{d}'[x_0y_0]+\mathrm{d}'[x_0](y-[y_0])+\mathrm{d}'(y-[y_0])[x_0]+\mathrm{d}'(x-[x_0])(y-[y_0])\\
     &=[(x_0y_0)^{1/p}]^{p-1}\mathrm{d}[(x_0y_0)^{1/p}]+\mathrm{d}\frac{[x_0](y-[y_0])}{p}+\mathrm{d}\frac{(x-[x_0])[y_0]}{p}+\mathrm{d}\frac{(x-[x_0])(y-[y_0])}{p}\\
     &=\sum_{Sym}\left([x_0][y_0^{1/p}]^{p-1}\mathrm{d}[y_0^{1/p}] +[x_0]\mathrm{d}\frac{y-[y_0]}{p} + \frac{x-[x_0]}{p}\mathrm{d}[y_0] + (x-[x_0])\mathrm{d}\frac{y-[y_0]}{p}\right)\\
     &= x\mathrm{d}'y+y\mathrm{d}'x
     \end{align*}
It follows that $\mathrm{d}'$ is a derivation. Then there exists $\delta\in \mathrm{End}_{W(R)}(\Omega_{W(R)/\mathbb{Z}_p})$ such that $\delta\mathrm{d}=\mathrm{d}'$, by the universal property of K\"{a}hler differential. It is very easy to check that $p\mathrm{d}'=d$. Then $\mathrm{d}'=\delta p \mathrm{d}'\Rightarrow \delta p=\mathrm{id}$. As $p$ commutes with $\delta$, $p\delta =\mathrm{id}$. Hence $\times p$ is an isomorphism.
  \end{proof}
\Cor
Let $R$ be as above, the map $$\sharp:\varprojlim_{\omega\rightarrow p\omega}\Omega_{W(R)/\mathbb{Z}_p}\rightarrow \Omega_{W(R)/\mathbb{Z}_p}, (\omega_0,\omega_1,\dots)\rightarrow \omega_0$$ is an isomorphism.
\begin{proof}
	Let $\delta$ be the inverse morphism of $\times p$ in the proof of last Lemma. Then it is easy to check that $\omega \mapsto (\omega,\delta\omega,\delta^2\omega,\dots)$ is the inverse morphism of $\sharp$.
\end{proof}
\Lemma
$\mathcal{O}_{K^\flat}$ is perfect of charateristic $p$.
\proof 
It is easy to check $(\ )^{1/p}:(\bar{x_0},\bar{x_1},\dots,)\mapsto (\bar{x_1},\bar{x_2},\dots)$ is the inverse morphism of Frobenius map.
     \Def["Differential Version" of Fontaine's Map]
    The following diagram :
    \begin{equation*}
    	\begin{tikzcd}
        A_{\text{inf}} \arrow
        {r}{\theta}&\mathcal{O}_K\\
        \mathbb{Z}_p\arrow {r}\arrow{u}&\mathcal{O}_{K_0}\arrow{u}
        \end{tikzcd}
    \end{equation*}
    commutes for the natural morphisms. Hence the diagram induces a natural morphism between their K\"{a}hler differential $\theta_*:\Omega_{A_{\text{inf}}/\mathbb{Z}_p} \rightarrow \Omega_{\mathcal{O}_K/\mathcal{O}_{K_0}}$. And also for $$\theta_*: \varprojlim_{\omega\mapsto p\omega}\Omega_{A_{\text{inf}}/\mathbb{Z}_p}\rightarrow \varprojlim_{\omega\mapsto p\omega}\Omega_{\mathcal{O}_K/\mathcal{O}_{K_0}}$$
    Then $$\theta_*\sharp^{-1}\mathrm{d}: A_{\inf} \rightarrow\varprojlim_{\omega\mapsto p\omega}\Omega_{\mathcal{O}_K/\mathcal{O}_{K_0}}$$
    is a derivation, denoted by $\mathrm{d}_{\theta}$, which is called the "differential version" of Fontaine's map.
    \Prop
    The "differential version" of Fontaine's map can be discribed in the following way:
    $$\mathrm{d}_\theta = (\mathrm{d}_{\theta,0},\mathrm{d}_{\theta,1},\dots)$$
    where
    \begin{align*}
    \mathrm{d}_{\theta,n} : A_{\text{inf}} &\rightarrow \Omega_{\mathcal{O}_K/\mathcal{O}_{K_0}}\\
    (x_0,x_1,\dots) & \mapsto \sum_{k=0}^{n-1}(x_k^{\sharp_{n-k}})^{p^{n-k}-1}\mathrm{d}x_k^{\sharp_{n-k}}+\mathrm{d}\theta((x_n^{1/p^n},x_{n+1}^{1/p^n},\dots))
    \end{align*}
    In particular, $\mathrm{d}_{\theta,n}([x^\flat]) = (\mathrm{d}x,x^{1-1/p}\mathrm{d}x^{1/p},x^{1-1/p^2}\mathrm{d}x^{1/p^2},\dots)$.
    \begin{proof}
    	If one denotes $$d_n(x_0,x_1,\dots) :=\sum_{k=0}^{n-1} [x_k^{1/p^{n-k}}]^{p^{n-k}-1}\mathrm{d}[x_k^{1/p^{n-k}}] + \mathrm{d}(x_n^{1/p^n},x_{n+1}^{1/p^n},\dots)$$ 
    	it will be easy to check that $d_0(x_0,x_1,\dots) = d(x_0,x_1,\dots)$ and $pd_n(x_0,x_1,\dots)=d_{n-1}(x_0,x_1,\dots)$. Hence $\sharp^{-1}\mathrm{d} = (d_0,d_1,\dots)$. Take the Fontaine's map on each term and notice that $\theta([x^{1/p^m}]) = x^{\sharp_m}$. Then one gets the conclusion.
    \end{proof}
    \Lemma
    The morphsim $\Omega_{\mathcal{O}_{\bar{K}_0}/\mathcal{O}_{K_0}}\otimes\mathcal{O}_K\rightarrow \Omega_{\mathcal{O}_K/\mathcal{O}_{K_0}}$ is injective.
    \begin{proof}
    	It is exactly the Proposition 8.18.
    \end{proof}
    \Prop[Proof of the Formula]\label{Prop.ProofofFormula}
    Let $S+T=1$, $A\in A_{\text{inf}}$ such that $[S^\flat]+[T^\flat]-1=([p^\flat]-p)A$, then for all $n\in\mathbb{N}$
    $$S^{1-1/p^n}\mathrm{d}S^{1/p^n}+T^{1-1/T^n}\mathrm{d}T^{1/p^n} =\theta(A)p^{1-1/p^n}\mathrm{d}p^{1/p^n}$$
    \begin{proof}
    	Use the "differential version" of Fontaine's map: 
    	\begin{align*}
        (S^{1-1/p^n}\mathrm{d}S^{1/p^n}+T^{1-1/T^n}\mathrm{d}T^{1/p^n})_{n\in\mathbb{N}}&=\mathrm{d}_\theta([S^\flat]+[T^\flat]-1)\\
        &=\mathrm{d}_{\theta}([p^\flat]-p)A\\
        &=\theta([p^\flat]-p)\mathrm{d}_{\theta}A+\theta(A)\mathrm{d}_{\theta}([p^\flat]-p)\\
        &=
        (\theta(A)p^{1-1/p^n}\mathrm{d}p^{1/p^n})_{n\in\mathbb{N}}
    	\end{align*}
    	Hence the formula holds in $\Omega_{\mathcal{O}_K/\mathcal{O}_{K_0}}$. By the Lemma above, it is also holds in $\Omega_{\mathcal{O}_{\bar{K}_0}/\mathcal{O}_{K_0}}\otimes\mathcal{O}_K$.
    \end{proof}
    	\Rmk
    	   If we choose a sequence $(a_0,a_1,\dots)$ in $\mathcal{O}_{\bar{K}_0}$ such that $a_n \equiv \theta(A)(\mathrm{mod}p^n)$, then the follow formula:
    	$$S^{1-1/p^n}\mathrm{d}S^{1/p^n}+T^{1-1/T^n}\mathrm{d}T^{1/p^n}=a_np^{1-1/p^n}\mathrm{d}p^{1/p^n}$$ holds in $\Omega_{\mathcal{O}_{\bar{K}_0}/\mathcal{O}_{K_0}}$. It is surprising that one can choose $a_n\in\mathbb{Z}[p^{1/p^{n-1}},T^{1/p^n},S^{1/p^n}]$ while  $\mathbb{Z}[p^{1/p^{n-1}},T^{1/p^n},S^{1/p^n}]$ may be much smaller than $\mathcal{O}_{K_0(p^{1/p^n},T^{1/p^n},S^{1/p^n})}$. We will give the proof of this Proposition in the following section from which one can also construct each $a_n$ in a computable way. 
    
    \section{The Problem about the Coefficients}
    In this section, we fix an element $B\in A_{\text{inf}}$ such that $\theta(B)=0$, and $A = B/([p^\flat]-p)$, $C =pA$. For any element $x\in A_{\text{inf}}$, $x_n$ always denotes the $n-$th coordinate of $x$ (i.e. $x = (x_0,x_1,\dots)$).
    \Def
    For an element $x\in A_{\text{inf}}$, a set $\mathfrak{R} = \{R_0,R_1,R_2,\dots\}$ of subring of $\mathcal{O}_K$ is called a system of coefficient rings associated to $x$ if one can choose $x_{i,j}\in R_j$ for each $i,j\in\mathbb{N}$ such that $x_i^{\sharp_j} \equiv x_{i,j} (\text{mod}\ p)$.
    \Rmk
    Let $\mathfrak{X}$ be a subset of $A_{\text{inf}}$, if $\mathfrak{R}$ is a common system of coefficient rings associated to each $x\in\mathfrak{X}$, so is $\mathfrak{R}$ to any $y\in \mathbb{Z}[\mathfrak{X}]$ (the subring generated by elements in $\mathfrak{X}$). For convenience, we will say that $\mathfrak{R}$ is associated to $x$, to $\mathfrak{X}$ and to $\mathbb{Z}[\mathfrak{X}]$ for short.
    \begin{proof}
    	It follows from $$(x_i+y_i)^{\sharp_j}\equiv S_i(x_{0,j},\dots,x_{i,j};y_{0,j},\dots,y_{i,j})(\text{mod} p)$$
    	$$(x_i\cdot y_i)^{\sharp_j}\equiv P_i(x_{0,j},\dots,x_{i,j};y_{0,j},\dots,y_{i,j})(\text{mod} p)$$
    	for any $x,y\in\mathfrak{X}$, where $S_i$ and $P_i$ are the polynomials associated to addition and multiplication for Witt vectors in Proposition 4.2.
    \end{proof}  
In the following part of this section, fix $\mathfrak{R} = \{R_0,R_1,R_2,\dots\}$ to be associate to $B$. More precisely, fix a  choice of $B_{i,j}\in\mathcal{O}_K$ for each $i,j\in\mathbb{N}$ such that $B_i^{\sharp_j}\equiv B_{i,j}(\text{mod}\ p)$, and then set $R_n :=\mathbb{Z}[B_{0,n},B_{1,n},\dots]$.
    \Def
    Let $m,k,N\in\mathbb{N}$, an element $a\in\mathcal{O}_{K^\flat}$ is called of type $(m,k,N,\mathfrak{R})$ if there exist some elements $f_n\in R_{n+k}[p^{1/p^{n+k-1}}]$ such that for any $n \geq N$, $a^{\sharp_n} \equiv \frac{f_n}{p^{m/p^{n}}}(\text{mod}\ p)$. We will call it of type $(m,k,N)$ for short.
    \Ex
    $A_0$ is of type $(1,1,0)$ as $A_0$ can be expressed as $(A_{0,n} := \frac{B_{0,{n+1}}^p}{p^{1/p^n}})_{n\in\mathbb{N}}$.
    \Rmk
    Here are some observation about the definition:
        \item[1] Assume $x$ is of type $(m,k,N)$, then it is also of type
        $$
        \begin{cases}
        	(m,k,N') & \text{if } N' \geq N\\
        	(m',k,N) & \text{if } (k\geq 1, m'\geq m) \text{ or } (m' \geq m, p|(m'-m))
        \end{cases}
        $$
        \item[2] Assume $x$ is of type $(m,k,N)$, then $x^p$ is of type $(pm,k-1,N+1)$.
        \item[3] Assume $x_0,\dots,x_n\in\mathcal{O}_{K^\flat}$ such that each $x_i$ is of type $(m,k,N)$. Then $\sum_{i=0}^{n}x_i$ is of type $(m,k,N)$.
        \item[4] Assume $x_0,\dots,x_n\in\mathcal{O}_{K^\flat}$ such that each $x_i$ is of type $(m_i,k,N)$. Then $\prod_{i=0}^{n}x_i$ is of type $(m,k,N)$, where $m = \sum_{i=0}^{n}m_i$.
    	\item[5] Assume $(p^\flat)^{p^N}x\in\mathcal{O}_{K^\flat}$ is of type $(m,k,N)$, then $x$ is of type $(p^N+m,k+1,N)$.
    \begin{proof}
    	\item[1] It is trivial.
    	\item[2] For any $n \geq N$ $$ (x^p)^{\sharp_{n+1}} = x^{\sharp_n} \equiv \frac{f_n}{p^{(pm)/p^{n+1}}}(\text{mod}\ p)$$ for some $f_n\in R_{n+k}[p^{1/p^{n+k-1}}]$. Hence $x^p$ is of type $(pm,k-1,N+1)$.
    	\item[3] Assume for any $j\geq N_i$, $x_i^{\sharp_j} \equiv \frac{f_{i,j}}{p^{m/p^j}}(\text{mod}\ p)$ for some $f_{i,j}\in R_{j+k}[p^{1/p^{j+k-1}}]$. Denote $x :=\sum_{i=0}^{n}x_i$, then for $j \geq N$,  
    	$$x^{\sharp_j}\equiv \sum_{i=0}^{n}\frac{f_{i,j}}{p^{m/p^j}}  = \frac{\sum_{i=0}^{n}f_{i,j}}{p^{m/p^j}}(\text{mod}\ p)$$
    	It follows that $x$ is of type $(m,k,N)$.
    	\item[4] Very similar to 2.
    	\item[5] Assume for any $n\geq N$, $((p^\flat)^{p^N}x)^{\sharp_n} = p^{p^{N-n}}x^{\sharp_{n}} \equiv \frac{f_n}{p^{m/p^n}}(\text{mod}\ p)$ for some $f_n\in R_{n+k}[p^{1/p^{j+k-1}}]$.
    	Then $$p^{p^{N-n}}x^{\sharp_n} = (p^{p^{N-n-1}}x^{\sharp_{n+1}})^p\equiv \frac{f_{n+1}^p}{p^{m/p^n}}(\text{mod}\ p^2)$$
    	As $N-n\leq 0$, $x^{\sharp_n} \equiv \frac{f_{n+1}^p}{p^{(p^N+m)/p^n}}(\text{mod}\ p)$. Hence $x$ is of type $(p^N+m,k+1,N)$.
    	
    	    \end{proof}
    \Lemma
    For the polynomial $S_n$ in Proposition 4.2, $S_n(X_0,X_1^p,\dots,X_n^{p^n};Y_0,Y_1,\dots Y_n^{p^n})$ is homogeneous of degree $p^n$. 
    \begin{proof}
    	Consider the equation:
    	$$S_n(X_0,X_1^p,\dots,X_n^{p^n};Y_0,Y_1,\dots Y_n^{p^n}) = X_n^{p^n}+Y_n^{p^n}+\sum_{i=0}^{n-1} p^{i-n}(X_i^{p^n}+Y_i^{p^n}-S_i^{p^{n-i}}(X_0,\dots,Y_i^{p^i}))$$ As 
    	 on the right hand side is homogeneous of degree $p^n$ by induction, $S_n(X_0,X_1^p,\dots,X_n^{p^n};Y_0,Y_1,\dots Y_n^{p^n})$ is homogeneous of degree $p^n$.
    \end{proof}
    \Cor
    For any $n\in\mathbb{N}^*$, $(p^\flat)^{p^n}A_n$ can be expressed as the following form:
    $$C_n+\sum_{i=1}^{n'}a_iB_0^{u_0(i)}\dots B_n^{u_n(i)}C_1^{v_1(i)}\dots C_{n-1}^{v_{n-1}(i)}$$
    for some $n',a_i\in\mathbb{N}$, such that for each $i\in\{1,2,\dots,n'\}$
    $$\sum_{j=0}^{n}p^j(u_j(i)+v_j(i)) = p^n$$ (here we set $v_0,v_n=0$).  
    \begin{proof}
    	As $[p^\flat]A =((p^\flat)A_0,\dots,(p^\flat)^{p^n}A_n,\dots)$ and $[p^\flat]A = B+C$,
    	it follows from Lemma 6.6.
    \end{proof}
    \Prop
    $A_n$ is of type $((n+1)p^n,1,n)$ for each $n\in\mathbb{N}$.
    \begin{proof}
    	As $p^\flat A_0 = B_0$, by observation 5($N = 0$), $A_0$ is of type $(1,1,0)$ while $B_0$ is of type $(0,0,0)$.\\
    	$C_n = A_{n-1}^p$ by Proposition 4.6. Assume the Proposition is true for $n < l$, then by observation 2,  $C_n$ is of type $(np^n,0,n)$ for $n\leq l$.Then for $n=l$, $(p^\flat)^{p^l}A_l$ can be  expressed as the following form:
    	$$C_l+\sum_{i=1}^{n'}a_iB_0^{u_0(i)}\dots B_l^{u_l(i)}C_1^{v_1(i)}\dots C_{l-1}^{v_{l-1}(i)}$$
    	such that for each $i\in\{1,2,\dots,n'\}$,  $\sum_{j=0}^{l}p^j(u_j(i)+v_j(i)) = p^l$. Each monomial is of type $(m_i,0,l)$ where $m_i = \sum_{j=1}^{l}v_j(i)jp^j \leq lp^l$. By observation 1, each monomial is of type $(lp^l,0,l)$, hence $(p^\flat)^{p^l}A_l$ is of type $(lp^l,0,l)$ by observation 3. Hence $A_l$ is of type $((l+1)p^l,1,l)$ by observation 5.
    \end{proof}
    \Def
     An element $a\in \mathcal{O}_{K}$ is called good w.r.t. $\mathfrak{R}$ if one can choose $x_n\in R_{n}[p^{1/p^{n-1}}]$ such that $x \equiv x_n(\text{mod}\ p^n)$ for each $n\in\mathbb{N}^*$. We will also call it good for short. 
     \Lemma
     For any $n,m\in\mathbb{N}$, assume $$(x_1+\dots+x_n)^{p^m} = \sum\limits_{i_1+\dots+i_n=p^m}a_{i_1,\dots,i_n}x_1^{i_1}\dots x_n^{i_n}$$
     Then for any $j\in\{1,\dots,n\}$, $v_p(a_{i_1,\dots,i_n}) + v_p(i_j) \geq m$.
     \begin{proof}
     	Notice that 
     	$$v_p(a!) = \sum_{i=0}^{\infty}[\frac{a}{p^i}]$$
     	then 
     	\begin{align*}
     		v_p(a_{i_1,\dots,i_n}) &= v_p(\frac{p^m!}{i_1!\dots i_n!})\\
     		&= \sum_{i=0}^{m}([\frac{p^m}{p^i}] - [\frac{i_0}{p^i}] -\cdots -[\frac{i_n}{p^i}])\\
     		& \geq m-v_p(i_j)
     	\end{align*}
     	for each $j$, as when $i>v_p(i_j)$, $[\frac{p^m}{p^i}] - [\frac{i_0}{p^i}] -\cdots -[\frac{i_n}{p^i}] \geq 1$. 
     \end{proof}
    \Prop
   $\theta(A)$ is good.
    \begin{proof}
    	As $\theta(A) = \sum_{n=0}^{\infty}p^n(A_n)^{\sharp_n}$, it is enough to prove each $p^n(A_n)^{\sharp_n}$ is good.\\
    	As $A_0^{\sharp_0} = \frac{1}{p}(B_0)^{\sharp_0} = -\sum_{i=1}^{\infty}p^{i-1}B_i^{\sharp_i}$, $A_0^{\sharp_0}$ is good (Notice that each $p^{i-1}(B_i)^{\sharp_i}$ is good).\\
    	Now assume for $n<l$, $p^n(A_n)^{\sharp_n}$ is good. For $n=l$, by Corollary 6.7, one can set $$(p^{\flat})^{p^l}A_l = C_l + \sum_{i=1}^{n'}a_iB_0^{u_0(i)}\dots B_l^{u_l(i)}C_1^{v_1(i)}\dots C_{l-1}^{v_{l-1}(i)}$$
    	such that for each $i$,  $\sum_{j=0}^{l}p^j(u_j(i)+v_j(i)) = p^l$. Now we denote $X_i := a_iB_0^{u_0(i)}\dots B_l^{u_l(i)}C_1^{v_1(i)}\dots C_{l-1}^{v_{l-1}(i)}$. Then for $s > l$, 
    	\begin{align*}
    		(\text{mod}\ p^s)p^l(A_l)^{\sharp_l}
    		&= p^{l-1}((p^\flat)^{p^l}A_l)^{\sharp
    		_l}\\
    	    &=
    	    p^{l-1}(C_l+\sum_{i=1}^{n'}X_i)^{\sharp_l}\\
    	    &\equiv p^{l-1}(C_l^{\sharp_s}+\sum_{i=1}^{n'}X_i^{\sharp_s})^{p^{s-l}}\\
            &= p^{l-1}C_l^{\sharp_l}+\sum\limits_{i_0+\dots+i_{n'} = p^{s-l} \atop
            	i_0 \neq p^{s-l}} b_{i_0,\dots,i_{n'}}(C_l^{i_0})^{\sharp_s}\prod_{j=0}^{l}(B_j^{u_j(i_1,\dots,i_{n'})})^{\sharp_s}\prod_{j=1}^{l-1}(C_j^{v_j(i_1,\dots,i_{n'})})^{\sharp_s}
    	\end{align*}
    	where
    	$$b_{i_0,\dots,i_{n'}} = p^{l-1}\frac{p^{s-l}!}{i_0!\dots i_{n'}!}a_0^{i_0}\dots a_{n'}^{i_{n'}}$$
    	$$u_j(i_1,\dots,i_{n'}) = i_1u_j(1)+\dots+i_{n'}u_j(n')$$
    	$$v_j(i_1,\dots,i_{n'}) = i_1v_j(1)+\dots+i_{n'}v_j(n')$$
    	$p^{l-1}C_l^{\sharp_l} = p^{l-1}A_{l-1}^{\sharp_{l-1}}$ is good by induction. The other monomials is of the form $$b(B_0^{\sharp_s})^{u_0}\dots (B_l^{\sharp_s})^{u_l}(C_1^{\sharp_s})^{v_1}\dots (C_l^{\sharp_s})^{v_l}$$ by the Lemma above, $v_p(b) + v_p(i_j)\geq s-1$ for each $j$, $v_p(b) + v_p(u_j),v_p(b) + v_p(v_j)  \geq s-1$. From
    	$$  (B_j)^{\sharp_s} \equiv B_{j,s}(\text{mod}\ p)$$
    	one gets
        $$
        	 (B_j^{\sharp_s})^{u_j}  \equiv B_{j,s}^{u_j}  (\text{mod}\ p^{s-v_p(b)})
        $$
        as $a^{p^n} \equiv b^{p^n} (\text{mod}\ p^{n+1})$ if $a \equiv b (\text{mod}\ p)$. As $C_j$ is of type $(jp^j,0,j)$, one can choose $C_{j,s}\in R_s[p^{1/{p^{s-1}}}]$ such that 
        $$(C_j)^{\sharp_s} \equiv \frac{C_{j,s}}{p^{jp^{j-s}}}(\text{mod}\ p)$$
    	Then 
    	$$(C_j^{\sharp_s})^{v_j} \equiv \frac{C_{j,s}^{v_j}}{p^{jv_jp^{j-s}}}(\text{mod}\ p^{s-v_p(b)})$$
    	Hence 
    	$$b(B_0^{\sharp_s})^{u_0}\dots (B_l^{\sharp_s})^{u_l}(C_1^{\sharp_s})^{v_1}\dots (C_l^{\sharp_s})^{v_l} \equiv \frac{b}{p^t}(B_{0,s})^{u_0}\dots (B_{l,s})^{u_l}(C_{1,s})^{v_1}\dots (C_{l,s})^{v_l} (\text{mod}\ p^s)$$
    	where $$t = \sum_{j=1}^{l}jv_jp^{j-s}$$
    	Consider the homogeneous condition (by Corallary 6.7, $\sum_{j=0}^{l}(u_j+v_j)p^j =p^s$), one gets $$\sum_{j=1}^{l}v_jp^j \leq p^s$$
    	Hence $$t = \sum_{j=1}^{l}jv_jp^{j-s} \leq l \text{\ and\ } t\leq l-1 \text{\ if\ }v_l=0$$ 
    	If $t\leq l-1$, $b/p^p\in\mathbb{Z}[p^{1/p^s-1}]$. And if $v_l\neq 0$, one also gets $b/p^t\in\mathbb{Z}[p^{1/p^{s-1}}]$ as $p^l|b$.  Hence each monomial is good, which implies $p^l(A_l)^{\sharp_l}$ is good. By induction, $\theta(A)$ is good.
    \end{proof}
	\Cor \label{Cor.Good}
	
	Let $R_n = \mathbb{Z}[T^{1/p^n},S^{1/p^n}]$, and $B = [T^\flat]+[S^{^\flat}] -1$, then $\theta(A)$ is good w.r.t. $(R_0,R_1,\dots)$. In other words, there exists a sequence $(a_1,a_2,\dots)$, where $a_n\in \mathbb{Z}[p^{1/p^{n-1}},T^{1/p^n},S^{1/p^n}]$ such that $$\theta(A) \equiv a_n(\text{mod}\ p^n)$$
	for each $n\in\mathbb{N}^*$
\begin{proof}
	For $p=2$, $-1 = (1,1,1,\dots)$ as $1 + (1,1,1,\dots) = [1]+[1]+2[1]+4[1]+\cdots = 0$. \\
	For $p\neq 2$, $-1 = (-1,0,0,\dots)$, as there is an natural morphism from $p$-type Witt vertors $W(\mathbb{Z})$ to $A_{\text{inf}}$ induced by $\mathbb{Z} \rightarrow \mathbb{Z}/p\mathbb{Z}\rightarrow \mathcal{O}_{K^\flat}$, and it is obvious that $[-1]+[1] =0$ in $W(\mathbb{Z})$.
	Hence $(R_0,R_1,\dots)$ is associated to $\{[T^\flat],[S^\flat],-1\}$, so by remark 6.2 and Proposition 6.11, $\theta(A)$ is good. 
\end{proof}
    \section{Some Computations} \label {Sect.Examples}
    
    In this section, all the differential forms are in $\Omega_{\mathcal{O}_{\bar{K}_0}/\mathcal{O}_{K_0}}$, and $a := ([S^\flat]+[T^\flat]-1)/([p^\flat]-p)$, here $S+T=1$.
    \Ex
    Denote $\alpha_{p^n} = (S^{1/p^n}+T^{1/p^n}-1)/p^{1/p^n}$, then 
    $$S^{1-1/p}\mathrm{d}S^{1/p}+T^{1-1/T}\mathrm{d}T^{1/p}=\alpha^p_pp^{1-1/p}\mathrm{d}p^{1/p}$$
    \begin{proof}
    	Denote $a=(a_0,a_1,\dots)$. As $([S^\flat]+[T^\flat]-1) = ([p^\flat]-p)a$, $a_0 = (S^\flat+T^\flat-1)/p^\flat$. Then $$\theta(a)\equiv a_0^{\sharp_0}\equiv(S^{1/p} +T^{1/p} -1)^p/p=\alpha_p^p (\mathrm{mod}p)  $$
    \end{proof}
    \Ex
    Denote $$\beta_p =\sum_{\mbox{\tiny$
    		\begin{array}{c}
    		i+j+k=p\\
    		i,j,k<p
    		\end{array}
    		$}}\frac{(p-1)!}{i!j!k!}(-1)^iT^{j/p^2}S^{k/p^2}+\delta_{2,p}+\alpha_{p^2}^p$$, then 
       $$S^{1-1/p^2}\mathrm{d}S^{1/p^2}+T^{1-1/T^2}\mathrm{d}T^{1/p^2}=(\alpha^{p^2}_{p^2}+\beta_p^p)p^{1-1/p^2}\mathrm{d}p^{1/p^2}$$
       Here $\delta_{2,p} =1$ if $p=2$, otherwise $=0$.\\ 
       
       In particular, if $p =2$, and denote $B = T^{1/4}S^{1/4}-T^{1/4}-S^{1/4}+1$,then
       $$S^{3/4}\mathrm{d}S^{1/4}+T^{3/4}\mathrm{d}T^{1/4} = (2\alpha_2^2+2\alpha_2B+(2\sqrt{2}+1)B^2) 2^{3/4}\mathrm{d}2^{1/4}$$
    \begin{proof}
    	By direct computation, denote $* :=\{(i,j,k)|i+j+k=p;i,j,k<p\}$
    	, one gets:
    	$$a_1 = \frac{\sum_{*}\frac{(p-1)!}{i!j!k!}
    				(-1)^i(T^\flat)^j(S^\flat)^k+\delta_{2,p}+(\frac{T^\flat+S^\flat-1}{p^\flat})^p}{(p^\flat)^p}$$
        Then 
        \begin{align*}
        	p(a_1)^{\sharp_1} & = ((p^\flat)^pa_1)^{\sharp 1}\\
        	&= (\sum_{*}\frac{(p-1)!}{i!j!k!}
        	(-1)^i(T^\flat)^j(S^\flat)^k+\delta_{2,p}+(\frac{T^\flat+S^\flat-1}{p^\flat})^p)^{\sharp_1}\\
        	&\equiv
        	(\sum_{*}\frac{(p-1)!}{i!j!k!}
        	(-1)^iT^{j/p^2}S^{k/p^2}+\delta_{2,p}+(\frac{T^\flat+S^\flat-1}{p^\flat})^{\sharp_1})^p\\
        	&\equiv (\sum_{*}\frac{(p-1)!}{i!j!k!}
        	(-1)^iT^{j/p^2}S^{k/p^2}+\delta_{2,p}+(\frac{T^{1/p^2}+S^{1/p^2}-1}{p^{1/p^2}})^p)^p\\
        	&= \beta_p^p (\mathrm{mod}p^2).
        \end{align*}    
        Hence 
        $$\theta(a) \equiv (a_0)^{\sharp_0}+p(a_1)^{\sharp_1} \equiv \alpha_{p^2}^{p^2}+\beta_p^p(\mathrm{mod}p^2).$$
        When $p=2$, notice that $\alpha_4^2=\alpha_2+\sqrt{2}B$, then
        \begin{align*}
        \alpha_4^4+\beta_2^2 &=
        \alpha_4^4+(B+\alpha_4^2)^2\\
        & = (\alpha_2+\sqrt{2}B)^2+(\alpha_2+\sqrt{2}B+B)^2\\
        & \equiv 2\alpha_2^2+2\alpha_2B+(2\sqrt{2}+1)B^2(\mathrm{mod}4).
        \end{align*}
        So
        $$S^{3/4}\mathrm{d}S^{1/4}+T^{3/4}\mathrm{d}T^{1/4} = (2\alpha_2^2+2\alpha_2B+(2\sqrt{2}+1)B^2) 2^{3/4}\mathrm{d}2^{1/4}.$$  
\end{proof} 

The Fontaine's map only gives a theoretical way to computer the coefficients, but the explicit expression will be complicate when $n$ is large. Now we try to set $S$ in a more general way and see what we could gain from the "Differrential Version" of Fontaine's Map.
\Ex For all  $s,t,u,v\in \mathbb{Z}$ st $gcd(u,v,p)=1$, let $S = (sT+t)/(uT+v)\in\mathcal{O}_{K_0}$, then there exists $a_{s,t,u,v}\in A_{\text{inf}}$ such that$$[S^\flat](u[T^\flat]+v)-(s[T^\flat]+t)=([p^\flat]-p)a_{s,t,u,v}$$
Then take $\mathrm{d}_\theta$ at both sides, one gets
$$(uT+v)S^{1-1/p^n}\mathrm{d}S^{1/p^n}+(uS-s)T^{1-1/p^n}\mathrm{d}T^{1/p^n} = \theta(a_{s,t,u,v})p^{1-1/p^n}\mathrm{d}p^{1/p^n}$$
\Rmk As $uT+v$ is invertible, $S^{1-1/p^n}\mathrm{d}S^{1/p^n}$ can be expressed as a linear combination of $T^{1-1/p^n}\mathrm{d}T^{1/p^n}$ and $p^{1-1/p^n}\mathrm{d}p^{1/p^n}$, which is also followed from $S^{1-1/p^n}\mathrm{d}S^{1/p^n}$ is a $[p^n]$-torsion and $({1-1/p^n}\mathrm{d}T^{1/p^n},p^{1-1/p^n}\mathrm{d}p^{1/p^n})$ is a basis of $p^n$-torsion part. As $(uS-s)=(ut-vs)/(uT+v)$, $T^{1-1/p^n}\mathrm{d}T^{1/p^n}$ can be expressed as a linear combination of $S^{1-1/p^n}\mathrm{d}S^{1/p^n}$ and $p^{1-1/p^n}\mathrm{d}p^{1/p^n}$ (in other words, $(S^{1-1/p^n}\mathrm{d}S^{1/p^n},p^{1-1/p^n}\mathrm{d}p^{1/p^n})$ is a basis of $p^n$-torsion part) if and only if $gcd(ut-vs,p)=1$, which is equal to say $S$ is not in the image of Frobenius map of $\mathbb{F}_p(T)$. This is exact the condition of Proposition 3.1.
\Ex
Assume $p\neq2$, denote $1^\flat := (1,\xi_p,\xi_{p^2},\dots)$, here $\xi_{p^n}$ is some primitive $p^n$-root of unit such $\xi_{p^n}^p=\xi_{p^{n-1}}$. Then there exists $a_{1^\flat}\in A_{\text{inf}}$ such that 
$$[1^\flat]-1=([p^\flat]-p )a_{1^\flat}$$ 
Take $\mathrm{d}_\theta$ at both side
$$\xi_{p^n}^{p^n-1}\mathrm{d}\xi_{p^n} = \theta(a_{1^\flat})p^{1-1/p^n}\mathrm{d}p^{1/p^n}$$
As $\theta(a_{1^\flat}) \equiv (a_{1^\flat,0})^{\sharp_0}\equiv (1-\xi_p)^p/p(\mathrm{mod}p)$, $v(\theta(a_{1^\flat})) =1/(p-1)$. It follows that $\xi_{p^n}^{p^n-1}\mathrm{d}\xi_{p^n}$ is a basis of $[p^{n-1/(p-1)}]$-torsion part of $\Omega_{\bar{\mathbb{Z}}_p\mathbb{Z}_p}$.
\Ex
Similarly, when $p=2$, denote $(-1)^\flat :=(-1,\xi_4,\xi_8,\dots)$. Then there exists $a_{(-1)^\flat}\in A_{\text{inf}}$ such that 
$$[(-1)^\flat]+1=([p^\flat]-p)a_{(-1)^\flat}$$
Take $\mathrm{d}_\theta$ at both side
$$\xi_{2^{n+1}}^{2^n-1}\mathrm{d}\xi_{2^{n+1}} =\theta(a_{(-1)^\flat})2^{1-1/2^n}\mathrm{d}2^{1/p^n}$$
As $\theta(a_{(-1)^\flat})\equiv(a_{(-1)^\flat})^{\sharp_0}\equiv (\xi_4+1)^2/2=\xi_4(\mathrm{mod}2)$. It follows that $\theta(a_{(-1)^\flat})$ is invertible and $ \xi_{2^{n+1}}^{2^n-1}\mathrm{d}\xi_{2^{n+1}}$ is a basis of $[2^n]$-torsion part of $\Omega_{\bar{\mathbb{Z}}_2/\mathbb{Z}_2}$.
\section{Appendix}
In this section, let $K_0$ be a complete valuation ring of charateristic 0 whose residue field is countable and of characteristic $p$. Let $K'$ be a finite extension of $K_0$ and denote $\bar{K}_0$ as the algebraic closure of $K_0$ and $K$ as the completion of $\bar{K}_0$. For an algebra $A$, denote $D(A)$ the derived category of $A$-modules. The aims of this section are to prove two things: \\
1. $H_1(\mathbb{L}_{\mathcal{O}_{K_0}/\mathcal{O}_{K'}}) = 0$.\\
2. $H_1(\mathbb{L}_{\mathcal{O}_{K}/\mathcal{O}_{\bar{K}}}) = 0$.\\
Here, for a ring map $A\rightarrow B$,  denote $\mathbb{L}_{A/B}$ the cotangent complex, whose precise definition can be refered to \cite[Def 3.2]{stack}.
And in \cite{cotangentcomplex}, there is a fairly explicit determination of $\tau_{\leq 2}\mathbb{L}_{B/A}$ which will be used to do computation in this section. Some of the definitions and Propositions will be listed.

\Def
Let $A\rightarrow B$ be a ring homomorphism. Choose a surjection $e_0:R \rightarrow B$ in $A$-algebras, where $R$ is a polynomial algebra over $A$, and let $I = \ker(e_0)$. Then choose an exact sequence 
$$0\rightarrow U\rightarrow F \xrightarrow{j} I \rightarrow 0$$ of $R$-modules, where $F$ is a free $R$-module, and defined $\varphi: F\otimes F \rightarrow F$ by
$$\varphi(x\otimes y) = j(x)y - j(y)x$$
Let $U_0\subset U$ be the image of $\varphi$, and let $e_1:F/U_0\rightarrow R$ and $e_2:U/U_0\rightarrow F/U_0$ be the induced maps. One will get a exact sequence of $A$-modules:
$$\varepsilon_{B/A}:0\rightarrow U/U_0\xrightarrow{e_2} F/U_0 \xrightarrow{e_1} R \xrightarrow{e_0} B \rightarrow 0$$ 
which is called a {\it free extension of $B$ over $A$}\cite[Def 2.1.3]{cotangentcomplex}.
\Def
Let $A,B$ and $\varepsilon_{B/A}$ be as above, {\it the cotangent complex of $\varepsilon_{B/A}$}\cite[Def 2.1.3]{cotangentcomplex} is defined as the following complex:
$$\mathbb{L}(\varepsilon_{B/A}): 0 \rightarrow U/U_0 \xrightarrow{d_2} F/U_0\otimes_R B\xrightarrow{d_1} \Omega_{R/A}\otimes_R B\rightarrow 0$$
where $d_2$ is induced by $e_2$ and $d_1 = (\mathrm{d}\otimes B)\circ(e_1\otimes B)$.

\Prop
The cotangent complex $\mathbb{L}(\varepsilon_{B/A})$ is independent to the choice of the free extension $\varepsilon_{B/A}$ up to homotopy \cite[2.1.9]{cotangentcomplex}, and there is a canonical map $\mathbb{L}_{B/A}\rightarrow\mathbb{L}(\varepsilon_{B/A})$ in $D(A)$, which induced an isomorphism $\tau_{\leq 2}\mathbb{L}_{B/A}\rightarrow\mathbb{L}(\varepsilon_{B/A}) $ in $D(B)$\cite[Lem 12.3]{stack}.

\Prop
$H_0(\mathbb{L}_{B/A})$ is naturally isomorphic to $\Omega_{B/A}$\cite[Lem 4.5]{stack}.

\Prop
Let $A\rightarrow B\rightarrow C$ be a sequence of ring homomorphisms. Then there is a canonical distinguished triangle 
 $$\mathbb{L}_{C/B}[-1] \rightarrow \mathbb{L}_{B/A}\otimes^{\mathbf{L}}_B C\rightarrow \mathbb{L}_{C/A} \rightarrow \mathbb{L}_{C/B}$$
in $D(C)$\cite[Prop 7.4]{stack}. In particular, the triangle above induced a long exact sequence
$$\cdots\rightarrow H_i(\mathbb{L}_{B/A}\otimes_B^{\mathbf{L}} C)\rightarrow H_i(\mathbb{L}_{C/A}) \rightarrow H_i(\mathbb{L}_{C/B})\rightarrow H_{i-1}(\mathbb{L}_{B/A}\otimes_{B}^{\mathbf{L}} C)\rightarrow\cdots$$
\Lemma

Consider a commutative square 
\begin{equation*}
	\begin{tikzcd}
	A' \arrow{r} & B' \\
	A \arrow{u}\arrow{r} & B \arrow{u} 
	\end{tikzcd}
\end{equation*}
of ring maps. If it induces a quasi-isomorphism $B\otimes^\mathbf{L}_A A' = B'$, then the functoriality map induces an isomorphism \cite[Lem 6.2]{stack} 
$$\mathbb{L}_{B/A}\otimes^{\mathbf{L}}_{B} B'\rightarrow \mathbb{L}_{B'/A'}$$ 

\Lemma
Let $A$ be a ring. For a element $f(X)\in A[X]$ such that $f'(X)$ is not a zero divisor in $A[X]/(f(X))$, then $$H_i(\mathbb{L}_{B/A}) = 0$$
for $i=1,2$, where $B := A[X]/(f(X))$.
\begin{proof}
	According to Definition 8.1, put $R := A[X]$ and $e_0$ be the natural surjection, $I := R(F(X))$, $F := R$ and $U:=0$. Hence the free extension is:
	$$\varepsilon_{B/A}: 0\rightarrow R \xrightarrow{\times f(X)} R \rightarrow B \rightarrow 0$$
	and the associated complex is:
	$$\mathbb{L}(\varepsilon_{B/A}): 0\rightarrow B\xrightarrow{d_1} \Omega_{R/A}\otimes_R B\rightarrow 0$$
	As $\Omega_{R/A}\otimes_R B = B\mathrm{d} X$, and $d_1(1) = \mathrm{d}f(X) = f'(X)\mathrm{d} X$, it follows that $d_1$ is an injection (note that $f'(X) $ is not a zero divisor in $B$). Hence $H_i(\mathbb{L}_{B/A}) = 0$ for $i = 1,2$.
\end{proof}

\Prop
Let $\mathcal{I}$ be a filtered index category and $M$ be a functor from $\mathcal{I}$ to $A$-algebras. If denote $B_i = M(i)$ for $i\in\mathcal{I}$, and $B =\varinjlim_\mathcal{I} B_i$. Then the induced map $$\varinjlim_\mathcal{I} H_n(\mathbb{L}_{B_i/A}) \rightarrow H_n(\mathbb{L}_{B/A})$$
is isomorphic.
\begin{proof}
	Refer to \cite[Lem 3.4]{stack}, set all $A_i$ be $A$. Note that the proof is actually for a filtered system even though the statement in \cite[Lem 3.4]{stack} is for a direct system.
\end{proof}

	\Prop
	Let $K_0$ be a complete discrete valuation field  of charateristic 0, with valuation ring $V$ and residue field $k$. The following assertions hold:
	\item 1. Let $L$ be the colimit of unramified extension of $K_0$, then $H_i(\mathbb{L}_{\mathcal{O}_L/V}) =0$ for $i=1,2$.
	\item 2.Assume k is separably algebraic closed, let $\{K_n\}_{n=1}^\infty$ be a sequence of finite extensions of $K_0$, such that for each $n$, $e(K_n|K_{n-1}) =1$ and $f(K_n|K_{n-1}) =p$. Let $E := \varinjlim K_n$, then $H_i(\mathbb{L}_{\mathcal{O}_E/V}) =0$ for $i=1,2$.
	\item 3.If $k$ is algebraic closed, let $\bar{K}_0$ be the algebraic closure of $K_0$ with valuation ring $\bar{V}$. Then $H_i(\mathbb{L}_{\bar{V}/V}) =0$ for $i=1,2$.  
	\begin{proof}
		For the case 1 and case 3, using Lemma 8.8, we only need to prove the situation for finite unramified and totally ramified extension $K'$ of $K_0$ and then take the colimit. Notice that in both cases, $\mathcal{O}_{K'}$ be expressed as the form:  $V[X]/f(X)$ with $f'(X) \neq 0$ by Lemma 1.6. Hence $H_i(\mathbb{L}_{\mathcal{O}_{K'}/V}) =0$ for $i=1,2$ by Lemma 8.7.\\
		For the case 2, as $[k(K_{n}):k(K_{n-1})] = p$, hence $k(K_{n})=k(K_{n-1})(\alpha_n)$ any $\alpha_n\in k(K_{n})\setminus k(K_{n-1})$. By Lemma 1.6 and Lemma 8.7 again, it follows that $H_i(\mathbb{L}_{\mathcal{O}_{K_n}/\mathcal{O}_{K_{n-1}}}) = 0$ for $i=1,2$. Note the exactness of the following sequence:
		$$\cdots\rightarrow H_i(\mathbb{L}_{\mathcal{O}_{K_{n-1}}/\mathcal{O}_{K}})\otimes_{\mathcal{O}_{K_{n-1}}}\mathcal{O}_{K_n}\rightarrow H_i(\mathbb{L}_{\mathcal{O}_{K_n}/\mathcal{O}_{K}})\rightarrow H_i(\mathbb{L}_{\mathcal{O}_{K_n}/\mathcal{O}_{K_{n-1}}})\rightarrow\cdots $$
		One will get $H_i(\mathbb{L}_{\mathcal{O}_{K_n}/V}) = 0$ for $i=1,2$ by induction.
	\end{proof}
    \Cor
    Let $\bar{K}_0$ be the algebraic closure of $K_0$, then $H_i(\mathbb{L}_{\mathcal{O}_{\bar{K}_0}/\mathcal{O}_{K_0}}) = 0$ for $i=1,2$.
    \begin{proof}
    	Fix an injection $i: K_0 \rightarrow \bar{K}_0$. First step, let $L$ be the filtered colimit of unramified extionsion of $K$, and by the case 1 of Proposition 8.9, it follows that $H_i(\mathbb{L}_{\mathcal{O}_{L}/\mathcal{O}_{K_0}}) = 0$ for $i=1,2$\\
    	Second step, denote $l_0$ the residue field of $L$ and construct $l_n := l_{n-1}(b_n)$ inductively: given a well-order $\{\leq,\bar{l_0}\}$ on $\bar{l_0}$ which is isomorphic to $(\leq,\mathbb{N})$ (note that $\bar{l_0}$ is countable as $l_0$ is countable), and $b_n$ is the smallest elements in $\bar{l_0}\setminus l_{n-1}$ such that is of degree $p$ over $l_{n-1}$. It is obvious that $\varinjlim l_n = \bar{l_0}$. As $l_n/l_{n-1}$ is purely inseparable extension of degree $p$, then for each $n$, there exists $c_{n-1}\in l_{n-1}$ such that $b_n^p =c_{n-1}$. Then construct $L_n$ inductively: choose a lifting $\tilde{c}_{n-1}$ of $c_{n-1}$ in $\mathcal{O}_{L_{n-1}}$ and a $p$-root $\tilde{b}_n\in\mathcal{O}_{\bar{K}}$ of $\tilde{c}_{n-1}$ arbitrarily, and let $L_n := L_{n-1}(\tilde{b}_n)$. Note $e(L_n|L_{n-1})=1$ and $k(L_n)=l_n$, it follows that $H_i(\mathbb{L}_{\mathcal{O}_E/\mathcal{O}_L})=0$ for $i=1,2$ according to the case 2 of Prop 8.9, where $E:= \varinjlim L_n$. \\
    	Third step, notice that $k(E)$ is already algebraic closed, it follows that $H_i(\mathbb{L}_{\mathcal{O}_{\bar{K}_0}/\mathcal{O}_E})=0$ for $i=1,2$ according to the case 3 of Prop 8.9.\\
    	Now we construct a sequence of ring homomophisms: $K_0\rightarrow L \rightarrow E\rightarrow \bar{K}_0$, and for any pair $(A,B)$ adjacent in this sequence, $H_1(\mathbb{L}_{\mathcal{O}_A/\mathcal{O}_B}) = 0 $ for $i=1,2$. One gets the conclusion.
    \end{proof}
    \Cor
    Let $K'/K_0$ be a algebraic extension of complete discrete valuation fields of characteristic 0. Then $H_1(\mathbb{L}_{\mathcal{O}_{K'}/\mathcal{O}_{K_0}}) = 0$
    \begin{proof}
    	Let $\bar{K}_0$ be the algebraic closure of $K'$. From the exactness of the sequence:
    	$$\cdots\rightarrow H_2(\mathbb{L}_{\mathcal{O}_{\bar{K}_0}/\mathcal{O}_{K'}})\rightarrow H_1(\mathbb{L}_{\mathcal{O}_{K'}/\mathcal{O}_{K_0}}\otimes_{\mathcal{O}_K'} \mathcal{O}_{\bar{K}_0}) \rightarrow H_1(\mathbb{L}_{\mathcal{O}_{\bar{K}_0}/\mathcal{O}_{K_0}})\rightarrow \cdots$$
    	one gets $H_1(\mathbb{L}_{\mathcal{O}_{K'}/\mathcal{O}_{K_0}}) = 0$ as $\mathcal{O}_{\bar{K}_0}$ is faithful flat over $\mathcal{O}_{K'}$.
    \end{proof}
    \Def
    Let $S$ be a scheme isomorphic to a scheme $\text{Spec} R$, where $R$ is a complete discrete valuation ring. An {\it S-variety} \cite[2.15]{Altenation} $X$ is an integral and separated scheme, flat and of finite type over $S$. We say that $X$ is {\it strictly semi-stable over $S$} \cite[2.16]{Altenation} if $X$ is locally smooth over the scheme $\text{Spec}R[t_1,\dots,t_r]/(t_1\cdot \dots \cdot t_r - \pi_R)$.
    \Def
    Let $X$ be a Noetherian integral scheme. An {\it alteration}\cite[2.20]{Altenation} $\tilde{X}$ of $X$ is an integral scheme $\tilde{X}$, together with a morphism $\varphi: \tilde{X} \rightarrow X$, which is dominant, proper and such that for some nonempty open $U\subset X$, the morphsim $\varphi^{-1}(U) \rightarrow U$ is finite.
    \Lemma
    Let $L$ be a complete discrete valuation field, and $\text{Spec} R$ be a $\text{Spec} \mathcal{O}_L$-scheme. If it is strictly semi-stable over $\text{Spec} \mathcal{O}_L$, then $$H_1(\mathbb{L}_{R/\mathcal{O}_L}) = 0$$
    \begin{proof}
    	w.l.o.g., we may assume there is a smooth ring map $$A := \mathcal{O}_L[t_1,\dots,t_r]/(t_1\cdot \dots \cdot t_r - \pi_L) \rightarrow R$$
        Hence $H_i(\mathbb{L}_{\mathcal{O}_L[t_1,\dots,t_{r-1}]/\mathcal{O}_L}) = 0 $\cite[Lem 4.7]{stack} and $H_i(\mathbb{L}_{R/A}) = 0$\cite[Lem 9.1]{stack} for $i\geq 1$. So it is enough prove that $H_1(\mathbb{L}_{A/ \mathbb{L}_{\mathcal{O}_L[t_1,\dots,t_{r-1}]}}) = 0 $ which is concluded immediately by Lemma 8.7.     
    \end{proof}
    The following Proposition is a simplified version of \cite[6.5]{Altenation}.
    \Prop(De Jong) If $X$ seperated scheme of finite type over $\mathcal{O}_{K_0}$, then there exists a finite extension $L/K_0$, a $\text{Spec}\mathcal{O}_L$-scheme $\tilde{X}$ which is strictly semi-stable over $\mathcal{O}_L$ and an alteration $ \tilde{X} \rightarrow X$ making the whole diagram
    \begin{equation*}
    	\begin{tikzcd}
\tilde{X} \arrow{r}\arrow{d} &X \arrow{d}\\
\text{Spec} \mathcal{O}_L \arrow{r} & \text{Spec}\mathcal{O}_{K_0}
    	\end{tikzcd}
    \end{equation*}
    commutative.
    \Lemma
    Let $K_0, \bar{K}_0$ and $K$ be as above, then $$H_1(\mathbb{L}_{K/K_0}) =0$$
    \begin{proof}
    	Let $R$ be a finitely generated $\mathcal{O}_{K_0}-$algebra, which is a sub-algebra of $\mathcal{O}_K$. Hence $X = \text{Spec} R$ is a seperated scheme of finite type over $\mathcal{O}_{K_0}$, by De Jong's alteration, we can construct the following commutative square
    	\begin{equation*}
    		\begin{tikzcd}
    		\tilde{X} \arrow{r}{\varphi}\arrow{d} &X \arrow{d}\\
    		\text{Spec} \mathcal{O}_L \arrow{r} & \text{Spec}\mathcal{O}_{K_0}
    		\end{tikzcd}
    	\end{equation*}
    	of ring maps, where $L$ is a finite extension of $K_0$, $\varphi$ is a alteration of $\text{Spec}\mathcal{O}_{K_0}$-scheme, and $\tilde{X}$ is strictly semi-stable over $\mathcal{O}_L$. Let $\text{Spec} R'$ be a open subset of $\tilde{X}$, hence $R'$ is a integral domain ($\tilde{X}$ is integral) and  $Frac(R')$ is finite over $Frac(R)$. It follows that $R'$ is algebraic over $Frac(R)$. Hence there is a morphism $\phi: \text{Spec} K \rightarrow \text{Spec} R' \hookrightarrow \tilde{X} $ such the following diagram 
    	\begin{equation*}
    		\begin{tikzcd}
    		\text{Spec} \mathcal{O}_K \arrow{r}  & X\\
    		\text{Spec} K \arrow{u}\arrow{r}{\phi} & \tilde{X} \arrow{u}{\varphi} 
    		\end{tikzcd}
    	\end{equation*}
    	is commutative. As $\varphi$ is proper, there exists a unique morphism $\text{Spec}\mathcal{O}_K \xrightarrow{\alpha} \tilde{X}$ making the whole diagram commutative. Choose an affine open subset $\text{Spec} \tilde{R}$ containing the image of $\alpha$, one gets a sequence of ring maps
    	$$ R \xrightarrow{\varphi} \tilde{R} \xrightarrow{\alpha} \mathcal{O}_K$$
    	If $\alpha$ is not an injection, choose $a\neq 0\in\ker(\alpha)$. As $a$ is algebraic over $Frac(R)$, then there exists a polynomial $f(X) = a_n X^n + \dots a_0 \in R[X]$ with $a_0 \neq 0$ such that $f(a) = 0$. It follows that $\alpha(a_0) = 0$. It is a contradiction. Hence $\tilde{R}$ is a sub-algebra of $\mathcal{O}_K$.\\
    	We claim that all the sub-algebra $\tilde{R}$ of $\mathcal{O}_K$ which are strictly semi-stable over some $\mathcal{O}_L$ ($L$ is some finite extension of $K_0$) are filtered and the colimit is $K$. Actually, for any $R_1$ and $R_2$ satisfied the condition, $R := R_1R_2$ is a finitely generated $\mathcal{O}_{K_0}$ algebra,  hence one can construct $\tilde{R}$ as above. The second assertion is trivial as any finitely generated $\mathcal{O}_{K_0}$ algebra can be contained in some $\tilde{R}$ satisfied the condition by the proof above.
    	Consider the sequence $\mathcal{O}_{K_0}\rightarrow \mathcal{O}_L \rightarrow \tilde{R}$ of ring maps, by Corollary 8.11, Lemma 8.14, it follows that $$H_1(\mathbb{L}_{\tilde{R}/\mathcal{O}_{K_0}})=0$$ According to Proposition 8.8, one gets $H_1(\mathbb{L}_{\mathcal{O}_K/\mathcal{O}_{K_0}}) = 0$. 
    \end{proof}
    \Lemma
    The morphism $$\times p: \mathbb{L}_{\mathcal{O}_{K}/\mathcal{O}_{\bar{K}_0}} \rightarrow \mathbb{L}_{\mathcal{O}_{K}/\mathcal{O}_{\bar{K}_0}}$$
    is an isomorphism in $D(\mathcal{O}_K)$. In particular, $H_i(\mathbb{L}_{\mathcal{O}_K/\mathcal{O}_{\bar{K}_0}}) $ is uniquely $p$-divisable for all $n\in\mathbb{N}$.
    \begin{proof}
    	It is easy to check the mapping cone is  $$Cone(\times p) = \mathbb{L}_{\mathcal{O}_{K}/\mathcal{O}_{\bar{K}_0}}\otimes^{\mathbf{L}}_{\mathcal{O}_{K}} ({\mathcal{O}_{K}}/p)$$ as $\cdots \rightarrow 0 \rightarrow {\mathcal{O}_{K}} \xrightarrow{\times p} {\mathcal{O}_{K}} \rightarrow 0 \rightarrow \cdots = ({\mathcal{O}_{K}}/p)[0]$ in $D({\mathcal{O}_{K}})$. Note that ${\mathcal{O}_{K}}/p = \mathcal{O}_{K} \otimes_{\mathcal{O}_{\bar{K}_0}}  {\mathcal{O}_{\bar{K}_0}}/p$, by Lemma 8.6, $Cone(\times p)$ is isomorphic to $\mathbb{L}_{(\mathcal{O}_K/p)/(\mathcal{O}_{\bar{K}_0}/p)}$, which is trivially equal to 0 as $(\mathcal{O}_K/p) = (\mathcal{O}_{\bar{K}_0}/p)$. Hence the induced map of $\times p$ on homology groups is isomorphic.  
    \end{proof}
    \Prop
    $H_1(\mathbb{L}_{\mathcal{O}_K/\mathcal{O}_{\bar{K}_0}}) = 0$. In particular, the canonical map
    $$\Omega_{\mathcal{O}_{\bar{K}_0/}/\mathcal{O}_{K_0}}\otimes_{\mathcal{O}_{\bar{K}_0/}} \mathcal{O}_K \rightarrow \Omega_{\mathcal{O}_{K}/\mathcal{O}_{K_0}}$$
    is an injection.
    \begin{proof}
    	As $H_1(\mathbb{L}_{\mathcal{O}_K/\mathcal{O}_{K_0}}) = 0$, the canonical map $H_1(\mathbb{L}_{\mathcal{O}_K/\mathcal{O}_{\bar{K}_0}})\rightarrow \Omega_{\mathcal{O}_{\bar{K}_0}/\mathcal{O}_{K_0}}\otimes_{\mathcal{O}_{\bar{K}_0}} \mathcal{O}_K$ is an injection. Note that any element in $\Omega_{\mathcal{O}_{\bar{K}_0}/\mathcal{O}_{K_0}}\otimes_{\mathcal{O}_{\bar{K}_0}} \mathcal{O}_K$ is $p^n$-torsion for some $n\in\mathbb{N}$. But $\times p^n$ is an isomorphism for $H_1(\mathbb{L}_{\mathcal{O}_K/\mathcal{O}_{\bar{K}_0}})$, hence $H_1(\mathbb{L}_{\mathcal{O}_K/\mathcal{O}_{\bar{K}_0}}) = 0$.
    \end{proof}

\end{document}